\documentclass[11pt]{amsart}

\usepackage{amsfonts}
\usepackage{amsthm, amssymb, ulem, amscd} 
\usepackage{amsmath}
\usepackage[mathscr]{euscript}

\def\crn#1#2{{\vcenter{\vbox{
        \hbox{\kern#2pt \vrule width.#2pt height#1pt
           }
          \hrule height.#2pt}}}}

\newcommand{\stopthm}{\hfill$\square$\medskip}

\newcommand{\pa}{\partial}
\newcommand{\na}{\nabla}
\newcommand{\Ric}{\operatorname{Ric}}

\newcommand{\Sym}{\operatorname{Sym}}

\newcommand{\Lo}{\mathring{L}}

\newcommand{\Span}{\operatorname{span}}

\newcommand{\tr}{\operatorname{tr}}

\newcommand{\R}{\mathbb R}

\newcommand{\C}{\mathbb C}

\newcommand{\ep}{\epsilon}

\newcommand{\al}{\alpha}

\newcommand{\ga}{\gamma}
\newcommand{\si}{\sigma}
\newcommand{\de}{\delta}
\newcommand{\be}{\beta}

\newcommand{\gh}{\widehat{g}}
\newcommand{\gb}{\overline{g}}
\newcommand{\hb}{\overline{h}}
\newcommand{\Hb}{\overline{H}}

\newcommand{\Pb}{\overline{P}}
\newcommand{\Wb}{\overline{W}}
\newcommand{\Xb}{\overline{X}}

\newcommand{\Yh}{\widehat{Y}}

\newcommand{\wt}{\widetilde}

\newcommand{\rh}{\widehat{r}}

\newcommand{\Ft}{\widetilde{F}}
\newcommand{\ft}{\widetilde{f}}

\newcommand{\cEb}{\underline{\mathcal{E}}}

\newcommand{\Ga}{\Gamma}

\newcommand{\Si}{\Sigma} 

\newcommand{\cL}{\mathcal{L}}
\newcommand{\cH}{\mathcal{H}}

\newcommand{\cHh}{\widehat{\mathcal{H}}}
\newcommand{\cF}{\mathcal{F}}
\newcommand{\cE}{\mathcal{E}}
\newcommand{\cK}{\mathcal{K}}

\newcommand{\cI}{\mathcal I}  
\newcommand{\cJ}{\mathcal J}  
\newcommand{\cC}{\mathcal{C}}

\newcommand{\cM}{\mathcal{M}}

\newcommand{\cU}{\mathcal{U}}
\newcommand{\cV}{\mathcal{V}}
\newcommand{\cW}{\mathcal{W}}

\theoremstyle{plain}
\newtheorem{theorem}{Theorem}[section]
\newtheorem{lemma}[theorem]{Lemma}
\newtheorem{proposition}[theorem]{Proposition}
\newtheorem{corollary}[theorem]{Corollary}

\theoremstyle{definition}

\theoremstyle{remark}
\newtheorem{remark}[theorem]{Remark}

\numberwithin{equation}{section}

\title[Higher-Dimensional Willmore Energies via Minimal Submanifold
  Asymptotics]{Higher-Dimensional Willmore Energies via Minimal Submanifold
Asymptotics}
\author[C. Robin Graham and Nicholas Reichert]{C. Robin Graham}
\address{Department of Mathematics, University of Washington,
Box 354350\\
Seattle, WA 98195-4350, USA}
\email{robin@math.washington.edu}

\author[]{Nicholas Reichert}
\email{nwr@uw.edu}

\begin{document}

\begin{abstract}
A conformally invariant generalization of the Willmore energy for compact
immersed submanifolds of even dimension in a Riemannian
manifold is derived and studied.  The energy arises as the coefficient of
the log term in the renormalized area expansion of a minimal submanifold 
in a Poincar\'e-Einstein space with prescribed boundary at infinity.  Its
first variation is identified as the obstruction to smoothness of the
minimal submanifold.  The energy is explicitly 
identified for the case of submanifolds of dimension four.  Variational
properties of this four-dimensional energy are studied in detail when the
background is a Euclidean space or a sphere, including identifications of
critical embeddings, questions of boundedness above and below for various
topologies, and second variation.   
\end{abstract}

\maketitle

\thispagestyle{empty}

\section{Introduction}\label{intro}

The Willmore energy $\int_\Si |H|^2 \,da_\Si$ of a compact immersed 
surface $\Si\subset\R^n$ measures
the total bending of the surface.  A basic property is its conformal
invariance.  In this paper we derive a conformally invariant generalization
of the Willmore 
energy for compact immersed submanifolds $\Si$ of even dimension $k\geq 2$
in a Riemannian manifold $(M,g)$ of dimension $n>k$.  Our energy $\cE$ is
defined via an inductive algorithm which for $k>4$ is prohibitively 
difficult to carry out to obtain explicit formulae.  However, we do derive
the formula for the $k=4$ energy, which we use to study some of its basic 
variational properties when $(M,g)$ is a Euclidean space or a sphere.        

Our energy arises upon consideration of a Plateau problem at infinity
for minimal submanifolds of dimension $k+1$ of an asymptotically
Poincar\'e-Einstein space $(X,g_+)$ of dimension $n+1$ whose boundary at   
infinity is $(M,g)$.  In case $(M,g)$ is a Euclidean space or a  
sphere, $(X,g_+)$ is the corresponding half-space or ball model of 
hyperbolic space.  Existence theory for minimal currents 
in the case that $g_+$ is hyperbolic is discussed in \cite{A1}, \cite{A2}.
Here we are concerned with formal asymptotics:   
we search for a submanifold $Y^{k+1}\subset \Xb$ satisfying 
$Y\cap M = \Si$, which is minimal to high order at infinity.  It turns out
that the minimality 
condition uniquely determines the Taylor expansion of $Y$ to order $k+2$,
at which point there is generically an obstruction $\cH\in \Gamma(N\Si)$ to
existence of a smooth $Y$.  Here $N\Si$ denotes the normal bundle to $\Si$
in $M$.

The area of any such asymptotically minimal $Y$ is infinite.  However,
finite quantities can be obtained by consideration of an asymptotic
expansion of the area.  One takes $X=M\times (0,\ep_0)_r$ near infinity and 
writes the Poincar\'e-Einstein metric in normal form relative to $g$ as 
$$
g_+=\frac{dr^2+g_r}{r^2},
$$
where $g_r$ is a 1-parameter family of metrics on $M$ satisfying $g_0=g$.  
The area of $Y\cap \{r>\ep\}$ has an asymptotic expansion in $\epsilon$,
and the generalized Willmore energy $\cE$ is defined to be the coefficient
of $\log\frac{1}{\epsilon}$ in the renormalized area expansion.  This turns 
out to be expressible as an integral over $\Si$ of a scalar invariant
of the local submanifold geometry of $\Si$ in $(M,g)$.  The integrand
involves derivatives of the second fundamental form of order up to
$(k-2)/2$.  

This area renormalization procedure was described in \cite{GW} and has
been studied in different contexts by various authors.  A main focus  
has been the constant term in the renormalized area expansion, usually
called the renormalized area.  When $k$ is odd, the renormalized area is a
global invariant of a minimal submanifold of a Poincar\'e-Einstein
space and there is no $\log\frac{1}{\epsilon}$ term in the expansion.  
The primary interest of \cite{GW} was the anomaly for the  
renormalized area when $k$ is even, measuring its failure to be 
conformally invariant under rescaling of $g$.  The paper \cite{GW} did
point out the conformal invariance of the coefficient of the 
$\log\frac{1}{\epsilon}$ term and identified it as the Willmore 
energy when $k=2$.  The contribution of the present paper, then, is to 
follow up with further analysis of this energy for $k>2$, particularly
from the point of view of regarding it as a generalization of the $k=2$
Willmore energy.  

As is well-known and described in \cite{Gr1}, there is an analogous
renormalization procedure for the volume of the asymptotically
Poincar\'e-Einstein manifold $(X,g_+)$ itself.  In this case, when $n$ is
even, the coefficient of the $\log\frac{1}{\epsilon}$ term in the expansion
is a conformal invariant of $(M,g)$ which equals a multiple of 
the integral of Branson's $Q$-curvature.  A basic result (\cite{HSS},
\cite{GH}) in this setting is that 
the metric variation of this coefficient is a multiple of the ambient 
obstruction tensor, which is a multiple of the coefficient of the first
$\log$ term in the expansion of $g_r$.  In Theorem~\ref{variation}, we
prove the analogous 
result for the generalized Willmore energy $\cE$: its variational
derivative with respect to variations of $\Si$ is the negative of the 
obstruction field $\cH$.  In particular, this identifies the Euler-Lagrange
equation for the energy $\cE$ as the equation $\cH=0$.  As a consequence,
in Proposition~\ref{minimalcritical} we deduce that if $(M^n,g)$ is
Einstein and $\Si$ is a minimal submanifold, then $\Si$ is critical for
$\cE$.  This 
generalizes a well-known property of the $k=2$ Willmore energy and produces
many examples of $\cE$-critical manifolds.  This can be viewed as an 
analogue in this setting of the fact that the ambient obstruction 
tensor vanishes for Einstein metrics.  

In Corollary~\ref{Eformcor}, we identify explicitly the $k=4$ energy $\cE$
for general background $(M^n,g)$ with $n\geq 5$.  When $M=\R^n$ with the 
Euclidean metric, our energy simplifies to:
\begin{equation}\label{Eeuclidean}
\cE = \frac{1}{128}\int_\Si \Big( |\nabla H|^2 
-|L^t H|^2 +\frac{7}{16}|H|^4 \Big) 
\,da_\Si.
\end{equation}
Here $L:S^2 T\Si\rightarrow N\Si$ is the second fundamental form
and $H=\tr L\in \Ga(N\Si)$ is the mean curvature vector.  
$L^t:N\Si\rightarrow S^2T\Si$ denotes the dual linear transformation and
$\nabla:\Ga(N\Si)\rightarrow \Ga(T^*\Si\otimes N\Si)$ the normal bundle
connection induced by the Levi-Civita connection of $g$.  This energy was
derived for $\Si^4\subset \R^5$   
in \cite{Gu2} by calculating how various quantities transform under  
conformal motions of $\R^5$ and searching for a linear combination which is 
conformally invariant.  However, that derivation dropped a factor of 
$-2$ in the calculations, so ended up with incorrect coefficients for 
$|L^tH|^2$ and $|H|^4$; compare \eqref{Eeuclidean} above with (61) of
\cite{Gu2}.  When $M=S^n$ with the round metric of sectional curvature 1,
our general formula reduces to:  
\begin{equation}\label{Esphere}
\cE = \frac{1}{128} \int_\Si \Big( |\nabla H|^2 -|L^t H|^2 
+\frac{7}{16} |H|^4 +6|H|^2 +48\Big) \,da_\Si.
\end{equation}

There are parallels that suggest that $\cE$ should be regarded 
as the correct $k=4$ analogue of the $k=2$ Willmore energy, and that
$\cE$-critical submanifolds of $\R^n$ or $S^n$ are analogues of  
Willmore surfaces.  We mentioned above that $\cE$ shares with the $k=2$
Willmore energy the fact that minimal submanifolds of $\R^n$ or $S^n$ are
critical.  For $k=4$ this is evident from the fact that all the 
nonconstant terms in the integrands of \eqref{Eeuclidean} and
\eqref{Esphere} are at least quadratic in $H$.  So, as examples of 
$4$-dimensional $\cE$-critical manifolds we have:  totally geodesic  
$S^4\subset S^n$, or any round $S^4\subset \R^n$, or the usual 
minimal 4-dimensional products of spheres in $S^n$.  The latter give rise 
to $\cE$-critical ``anchor rings'' in Euclidean spaces via stereographic
projection.  In
light of the interest of the Willmore conjecture, it is natural to ask
about the behavior of $\cE$ as one varies over immersions of $\Si^4$ having
the topologies of these $\cE$-critical anchor rings.  We show:
\begin{proposition}\label{no}
$\cE$ is unbounded above and below over embeddings of any of the following: 
\begin{gather*}
S^2\times S^2 \subset S^5  \\
S^1\times S^3\subset S^5 \\
S^1\times S^1\times S^2\subset S^6 \\
S^1\times S^1\times S^1\times S^1\subset S^7. 
\end{gather*}
\end{proposition}

\noindent
Proposition~\ref{no} rules out the most naive formulations of a
4-dimensional Willmore Conjecture for $\cE$.  As we
discuss in \S\ref{otherenergies} (see also \cite{Gu2}), it is possible to
modify $\cE$ to obtain non-negative 
conformally invariant energies by adding appropriate multiples of the 
fourth power of the norm of the trace-free 
second fundamental form.  But then one loses the property that minimal
submanifolds of $\R^n$ and $S^n$ are critical.  This seems to us an
important geometric property of a higher-dimensional energy to be regarded
as a true analog of the Willmore energy.       

We do not know whether $\cE$ is bounded below over immersions of
$S^4$ in $\R^5$.  It is unbounded above:  we used {\it
Mathematica} to calculate $\cE$ explicitly for the family of ellipsoids in
$\R^5$ with axes of length $(1,1,1,1,a)$ with $a>0$.  {From} this one can 
deduce that $\cE\rightarrow \infty$ as $a\rightarrow 0$ and as
$a\rightarrow \infty$.  Moreover, a numerical plot suggests that $\cE$ is
convex as a function of $a$ and has a unique minimum at $a=1$,
corresponding to the round $S^4$. 
As far as we know, it is plausible that a round $S^4\subset \R^5$ minimzes
$\cE$ over all immersions of $S^4$ in $\R^5$; it would be interesting to
resolve this question.  Locally this is the case:  

\begin{proposition}\label{2varS4}
The second variation of $\cE$ at a round $S^4\subset \R^5$ is nonnegative,
and is positive in directions transverse to the orbit of the conformal
group.   
\end{proposition}

It might be interesting to study the boundedness properties of $\cE$ 
over immersions of $S^1\times S^1\times S^2$ or $\big(S^1\big)^4$ in $S^5$.   
Likewise, to study variational properties of $\cE$ for $\Si$ with other 
topologies, for instance, $\Si = \C\mathbb{P}^2$.  
The Veronese embedding $\C\mathbb{P}^2\rightarrow S^7$ is minimal; hence 
$\cE$-critical.  We would like to think, without real concrete evidence,
that there should be some interesting variational problems for the energy
$\cE$, perhaps of min-max type.  

Gover and Waldron have developed a program to study conformal hypersurface
geometry based on the the singular Yamabe problem and tractor calculus
(\cite{GoW1}, \cite{GoW2}, \cite{GoW3}, \cite{GGHW}, \cite{GoW5}).  This
includes the derivation of conformally invariant energies for
hypersurfaces in both parities of dimension.  These energies were made
explicit in terms of underlying geometry for hypersurfaces of dimension 2
and 3.  In \cite{V}, Vyatkin used 
tractor methods to derive an explicit conformally invariant energy for 4
dimensional 
hypersurfaces in conformally flat 5-manifolds.  In \S\ref{otherenergies},
we relate his energy to $\cE$ in the case when the background manifold is
$\R^5$.  Conformally invariant hypersurface energies based on volume
renormalization of singular Yamabe metrics are defined and analyzed in
\cite{Gr2}, \cite{GoW4}.    

The paper \cite{Z} of Y. Zhang was posted while this paper was in final
preparation.  
Zhang also studies the expansion of minimal submanifolds of
Poincar\'e-Einstein spaces through the terms of order 4, derives a formula
for $\cE$ for $k=4$, and obtains the critical points we
derive in \S\ref{prod} for products of spheres in Euclidean spaces.     

This paper is organized as follows.  
In \S\ref{formulation}, we review the formal asymptotics of 
minimal submanifolds of Poincar\'e-Einstein spaces up to the 
order of the locally undetermined term in the expansion.
The treatment in
\cite{GW} was incomplete in that it derived non-invariant asymptotic
expansions in local coordinates but provided no explanation 
of how to formulate the results globally.  Such a global formulation is   
needed to construct the renormalized area expansion.  In
Theorem~\ref{U}, we provide an invariant  
formulation of the asymptotics using the normal exponential map of the
boundary submanifold.  All the coefficients in the expansion are
invariantly defined sections of the normal bundle of $\Si$ determined by
its geometry as a submanifold of $(M,g)$.  We conclude \S\ref{formulation}
by showing that the obstruction field $\cH\in \Gamma(N\Si)$, which arises
as the coefficient of the first log term in the expansion, is invariant
under conformal rescalings of $g$.   

In \S\ref{energysection}, we consider the renormalized area expansion,
define  
the energy $\cE$, and show that it is conformally invariant.  We then prove
that the first variation of $\cE$ is the negative of the obstruction field 
$\cH$, and deduce that a minimal $\Si$ is $\cE$-critical if $(M,g)$ is
Einstein.  

In \S\ref{formulas}, we derive formulas for the 
expansion of the minimal submanifold and the renormalized area  
through order 
4.  This gives formulas for $\cE$ for $k=2$, $4$ and 
and for $\cH$ for $k=2$.  We use a formalism of Guven to identify 
$\cH$ for $k=4$ when the background is a Euclidean space.  
These formulas for $\cH$ are in particular the negatives of formulas 
for the first variation of $\cE$ in the cases $k=2$, $4$.   
We conclude \S\ref{formulas} by commenting on the nonzero  
leading terms in the expansion coefficients, the obstruction field, and the
energy.   

In \S\ref{4d}, we analyze $\cE$ when $\dim\Si=4$.   In \S\ref{prod} we
calculate explicitly the energy of products of spheres in $S^n$ as a
function of the 
radii, and use the resulting formulas to identify $\cE$-critical embeddings  
of products of spheres.  This gives examples of some non-minimal
$\cE$-critical embeddings.  
We also make a remarkable observation
concerning the relationships between these energy formulas for the
different topologies.  
In \S\ref{anchor}, we stereographically   
project products of spheres to obtain 4-dimensional anchor rings in
Euclidean spaces, thereby obtaining $\cE$-critical anchor rings.  We also
analyze the energy of a non-isotropically dilated family of such anchor
rings to show that $\cE$ is unbounded above over embeddings of $S^2\times
S^2$ in $S^5$.  This combined with the results in \S\ref{prod} enable us  
to prove Proposition~\ref{no}.  In \S\ref{2var}, we derive a general
formula for the second variation of $\cE$ at a minimal immersed
hypersurface in $S^5$, precisely generalizing the corresponding
formula derived in \cite{W} for the classical Willmore energy.  We apply
this formula to $S^4\subset S^5$, thereby proving Proposition~\ref{2varS4},
and also to the standard minimal embedding $S^2\times S^2\subset S^5$. 
Finally, in \S\ref{otherenergies} we discuss other energies obtained by
modifying $\cE$ by adding conformally invariant expressions.  
In particular, we construct non-negative energies and we derive the
relationship mentioned above between $\cE$ and Vyatkin's energy.  

\bigskip
\noindent
{\it Acknowledgements.}  Research of CRG was partially supported by NSF
grant \# DMS 1308266.  NR held a postdoctoral position at the University of
Washington while this work was carried out.  His research was also
supported by NSF RTG grant \#  DMS 0838212 and a PIMS Postdoctoral
Fellowship.  He is grateful to all these institutions for their support.  

\section{Notation and Conventions}\label{notation}
For a Riemannian manifold $(M^n,g)$, we denote the Levi-Civita connection 
by ${}^M\nabla$, the curvature tensor by 
$R_{ijkl}$, the Ricci tensor by $\Ric(g)$ or $R_{ij}=R^k{}_{ikj}$, and the
scalar curvature by $R=R^i{}_i$.  Our sign convention for $R_{ijkl}$ is
such that spheres have  
positive scalar curvature.  $S^n(r)$ denotes the Euclidean sphere of
dimension $n$ and radius $r$, and the notation $S^n$ is used for 
$S^n(1)$.  The Schouten tensor of $(M,g)$ is 
$$
P_{ij}=\frac{1}{n-2}\Big(R_{ij}-\frac{R}{2(n-1)}g_{ij}\Big) 
$$
and the Weyl tensor is defined by the decomposition
\begin{equation}\label{curv}
R_{ijkl}=W_{ijkl}+P_{ik}g_{jl}-P_{jk}g_{il}-P_{il}g_{jk}+P_{jl}g_{ik}.  
\end{equation}
The Cotton and Bach tensors are
$$
C_{ijk}={}^M\nabla_kP_{ij}-{}^M\nabla_jP_{ik}
$$
and 
$$
B_{ij}={}^M\nabla^kC_{ijk}-P^{kl}W_{kijl}. 
$$
In invariant expressions such as these, each Latin index $i$, $j$, $k$ can
be interpreted as a label for $TM$ or its dual.  

$\Si$ will denote an immersed submanifold of $(M,g)$ of dimension $k$ via
an immersion 
$f:\Si\rightarrow M$.  The pullback bundle $f^*TM$ splits as
$f^*TM=T\Si\oplus N\Si$.  We use $\al$, $\be$, $\ga$ as index labels for
$T\Si$ and $\al'$, $\be'$, $\ga'$ for $N\Si$.  A Latin index $i$ for an
element or section of $f^*TM$ thus corresponds to a pair $(\al,\al')$.  So,
for instance, when restricted to $\Si$, the Schouten tensor $P_{ij}$ splits 
into its tangential $P_{\al\be}$, mixed $P_{\al\al'}$, and normal
$P_{\al'\be'}$ pieces.  Likewise, the restriction of the metric $g_{ij}$ to
$\Si$ can be identified with the metric $g_{\al\be}$ induced on $\Si$
together with the metric $g_{\al'\be'}$ induced on $N\Si$.  We routinely
use $g_{\al\be}$ and $g_{\al'\be'}$ and their inverses to lower and raise  
unprimed and primed indices.

The second fundamental form is $L:S^2T\Si\rightarrow N\Si$, defined by  
$L(X,Y)=({}^M\na_X Y)^\perp$.  We typically write it as
$L_{\al\be}^{\al'}$, or perhaps as $L_{\al\be\al'}$ or
$L_\al{}^\be{}_{\al'}$ upon lowering and/or raising indices.  Since $L$ has 
only one primed index and 
is symmetric in $\al\be$, it is not necessary to pay 
attention to the order of the three indices.  The mean curvature vector is 
$H=\tr L$, i.e. the section of $N\Si$ given by   
$H^{\al'}=g^{\al\be}L_{\al\be}^{\al'}=L_{\al}{}^{\al\al'}$.   

The Levi-Civita connection on $M$ induces connections on $T\Si$ and 
$N\Si$ together with their duals and tensor products, all of which we
denote $\nabla$.  So, for instance, we can form the covariant
derivative $\na_\al H^{\al'}$, which is a section of $T^*\Si\otimes N\Si$.
A point which requires 
some attention is that if we have a tensor on $M$ defined near $\Si$ (such
as the Schouten tensor $P_{ij}$), we can form its covariant derivative 
${}^M\na_kP_{ij}$ and then consider on $\Si$ a piece of this tensor such  
as ${}^M\na_\al P_{\al'\be}$.  Alternately, we can first consider on $\Si$
the piece  
$P_{\al'\be}$, which is a section of $N^*\Si\otimes T^*\Si$, and then
differentiate with respect to the induced connection to obtain $\na_\al
P_{\al'\be}$.  Of 
course, these are different sections of $T^*\Si\otimes N^*\Si\otimes
T^*\Si$; the distinction is indicated by the ${}^M$ specifying the
connection used.     

Norms are always taken with respect to the metric on tensor products
induced by the metric on the underlying bundle.  So, for instance, 
in \eqref{Eeuclidean}, we have 
\[
\begin{gathered}
|\na H|^2 = \na_\al H^{\al'}\na^\al H_{\al'}\\
|L^t H|^2 = L_{\al\be}^{\al'}L^{\al\be}_{\be'}H_{\al'}H^{\be'}\\
|H|^4 = \big( H_{\al'}H^{\al'}\big)^2.
\end{gathered}
\]

We often compute in local coordinates.  We always use a  
coordinate system $\{(x^\al,u^{\al'}):1\leq \al\leq k, 1\leq \al'\leq  
n-k\}$ for $M$ near $\Si$, with the properties that 
$\Si=\{u^{\al'}=0\}$ and $\pa_\al\perp \pa_{\al'}$ on $\Si$.  Hence, on
$\Si$, the 
$\pa_\al$ span $T\Si$, the $\pa_{\al'}$ span $N\Si$, and the mixed 
metric components $g_{\al\al'}$ vanish.  So our use of indices for
coordinates is consistent with the abstract interpretation described
above.  When computing in local coordinates, partial derivatives are
expressed using either of the two notations $\pa_\al u_\be = u_{\be,\al}$.   

Our sign convention for Laplacians is that $\Delta=\sum \pa_i^2$ on
Euclidean space.  

When dealing with embedded submanifolds, as in \S\ref{formulation}, we 
typically identify $\Si$ with its image and suppress mention of the
immersion $f$.

\section{Formal Asymptotics}\label{formulation}

Let $(M^n,[g])$ be a conformal manifold, $n\geq 2$, and $g$ a chosen metric
in the conformal class.  By a Poincar\'e metric $g_+$ in normal
form relative to $g$, we will mean a metric $g_+$ on $X=M\times (0,\ep_0)_r$,
for some $\ep_0>0$, of the form  
\begin{equation}\label{normalform}
g_+=\frac{dr^2+g_r}{r^2},
\end{equation}
where $g_r$ is a smooth 1-parameter family of metrics on $M$ for which
$g_0=g$, and satisfying
$$
\Ric(g_+)+ng_+ = O(r^{n-2}).
$$
These conditions uniquely determine the Taylor expansion of $g_r \mod
O(r^n)$, and it is even to this order (\cite{FG}).  The form of the
expansion changes  
at order $n$ for solutions to higher order, but that will not be relevant
here.  Set $\gb=r^2g_+=dr^2+g_r$.  We identify $M$ with $M\times \{0\}$,
and view $M=\pa X$ as the 
boundary at infinity relative to $g_+$.   In case $M=\R^n$ and $g=|dx|^2$
is the Euclidean metric, $g_r=g$ is constant in $r$, and $g_+$ is the
upper-half space realization of the hyperbolic metric.  

In this section we consider local geometry of embedded submanifolds of $M$.
In the next section we will 
construct global invariants of immersed submanifolds by integration of
the local invariants derived here.  

Let $\Sigma\subset M$ be an embedded submanifold of
dimension $k$, $2\leq k<n$, with $k$ even.  We consider the formal  
asymptotics of embedded submanifolds $Y^{k+1}\subset \Xb=M\times 
[0,\ep_0)$ with $\pa Y =\Sigma$ which are minimal with respect to $g_+$. 
Such a submanifold can be described invariantly in terms of a 1-parameter
family of sections of the $g$-normal bundle $N\Sigma$ of $\Si$ in $M$ as
follows.     

The normal 
exponential map of $\Si$ with respect to $g$, denoted $\exp_\Si$, defines a 
diffeomorphism from a neighborhood of the zero  
section in $N \Sigma$ to a neighborhood of $\Sigma$ in $M$.  Let 
$Y^{k+1}\subset M\times [0,\ep_0)$ be a smooth submanifold which is
transverse to $M$ and satisfies $Y\cap M = \Sigma$.  For $r\geq 0$ small,
let $Y_r\subset M$ denote the slice of $Y$ at height $r$, 
defined by $Y\cap (M\times \{r\})=Y_r\times \{r\}$.  Then $Y_r$ is a smooth 
submanifold of $M$ of dimension $k$ and $Y_0=\Sigma$.  There is a unique
section $U_r\in\Gamma(N\Sigma)$ so that $\exp_\Si \{U_r(p):p\in\Si\} =
Y_r$.   This defines a smooth 1-parameter family $U_r$ of
sections of $N \Sigma$ for which 
\begin{equation}\label{Y}
Y=\left\{\big(\exp_\Si U_r(p),r\big): p\in \Sigma, r\geq 0\right\}.
\end{equation}
In particular, $U_0=0$.  The submanifolds $Y\subset X$ which we consider 
will all be 
orthogonal to $M$ at $\Sigma$ with respect to $\gb$.  Thus the tangent 
bundle to $Y$ along $\Sigma$ is 
$T\Sigma\oplus \Span{\pa_r}$, and the normal bundle to $Y$ along $\Si$ can
be identified with $N\Si$.  In this case we have $\pa_r U_r|_{r=0}=0$,  
i.e. $U_r = O(r^2)$.  

The condition that $Y$ is minimal becomes a system of partial differential
equations on the normal vector fields $U_r$.  Recall that minimality of $Y$
is equivalent to the statement that $H_Y=0$, where $H_Y$ denotes the mean
curvature vector of $Y\subset X$ with respect to the metric $g_+$.

\begin{theorem}\label{U}
There is a smooth $U_r$ so that $|H_Y|_{\gb}=O(r^{k+2})$.  This condition
uniquely determines the Taylor expansion of $U_r$ modulo $O(r^{k+2})$, and
this Taylor expansion is even in $r \mod O(r^{k+2})$ .  The quantity  
$\cH:=r^{-k-2}H_Y|_{r=0}$   
defines a section of $N \Sigma$ which is independent of the choice of the
$O(r^{k+2})$ ambiguity in $U_r$.  If nonzero, $\cH$ is therefore an
obstruction to solving $|H_Y|_{\gb}=o(r^{k+2})$ with $U_r$ a formal power
series.   

There is a solution to $|H_Y(U_r)|_{\gb} = O(r^{k+3}|\log r|)$ of the form 
\begin{equation}\label{Uwithlog}
U_r = V_r -(k+2)^{-1}\cH r^{k+2} \log r, 
\end{equation}
where $V_r$ is smooth.  The $r^{k+2}$ coefficient in $V_r$ is  
formally undetermined.   
\end{theorem}

\begin{remark}
The same result is true for $k$ odd, but in that case
$\cH$ is always identically zero.
\end{remark}

\begin{remark}
Boundary regularity for minimal hypersurfaces in hyperbolic space has been
studied in \cite{HL}, \cite{L1}, \cite{L2}, \cite{T}, \cite{HJ}, and
\cite{HSW}.  
\end{remark}

\begin{proof}
A local coordinate version of this result was derived in \S 2 of \cite{GW}.  
We 
show how to reformulate Theorem~\ref{U} in terms of local coordinates
and outline the proof, referring to \cite{GW} for some details.     

We will work in geodesic normal coordinates on 
$M$ near $\Sigma$.  Choose a local coordinate system $\{x^\al: 1\leq 
\al\leq k\}$ for an open subset $\cV\subset \Sigma$ and a local frame
$\{e_{\al'}(x): 1\leq \al'\leq  n-k\}$ for $N\Sigma|_{\cV}$.  Let 
$\{u^{\al'}: 1\leq \al'\leq n-k\}$ denote the    
corresponding linear coordinates on the fibers of $N \Sigma|_{\cV}$.  The
map $\exp_\Si \big(u^{\al'} e_{\al'}(x)\big)\mapsto (x,u)$ defines a  
coordinate system $(x^\al,u^{\al'})$ in a neighborhood $\cW$ of $\cV$ in
$M$, with respect to which $\Sigma$ is given by $u^{\al'}=0$.  For each
$(x,u)$, 
the curve $t\mapsto (x,tu)$ is a geodesic for $g$ normal to $\Sigma$.  In 
particular, in these coordinates the mixed metric components $g_{\al\al'}$
vanish on $\cV$.  Extend the coordinates to $\cW\times [0,\ep)\subset \Xb$  
to be constant in $r$.  If $U_r$ is a 1-parameter family of sections of $N
\Sigma$ and we define $u^{\al'}(x,r)$ by $U_r(x) =
u^{\al'}(x,r)e_{\al'}(x)$, then the description \eqref{Y} of $Y$ is the
same as saying that in the coordinates $(x,u,r)$ on $X$, $Y$ is the graph
$u^{\al'}=u^{\al'}(x,r)$.     

The setting in \cite{GW} was
that $(x^\al,u^{\al'})$ is any local coordinate system on $M$ near a point
of $\Si$ with the properties that $\Si = \{u^{\al'}=0\}$ and
$\pa_{\al}\perp \pa_{\al'}$ on $\Si$, and $Y$ is described as the graph  
$u^{\al'}=u^{\al'}(x,r)$.  So our geodesic normal coordinates constructed
above and our description of $Y$ in terms of them are of this form.  

Let $h$ denote the metric induced on $Y$ by $g_+$ and set $\hb = r^2h$,
so $\hb$ is the metric induced by $\gb=dr^2+g_r$.  Now $(x^\al,r)$ restrict
to local coordinates on $Y$.  In terms of these coordinates, $\hb$ is given
by: 
\begin{equation}\label{hbar}
\begin{split}
\hb_{\al\be}=&g_{\al\be}+2g_{\al'(\al} u^{\al'}{}_{,\be)}
+g_{\al'\be'}u^{\al'}{}_{,\al}u^{\be'}{}_{,\be} \\
\hb_{\al 0} =&g_{\al\al'}u^{\al'}{}_{,r}
+g_{\al'\be'}u^{\al'}{}_{,\al}u^{\be'}{}_{,r}\\ 
\hb_{00}=& 1+g_{\al'\be'}u^{\al'}{}_{,r}u^{\be'}{}_{,r}.
\end{split}
\end{equation}
We use a '$0$' index for the $r$-direction.  
The components of $\hb$ and the derivatives of $u$ are evaluated at
$(x,r)$.  We have written 
$$
g_r=g_{\al\be}(x,u,r)dx^{\al}dx^\be
+2g_{\al\al'}(x,u,r)dx^\al
du^{\al'}+g_{\al'\be'}(x,u,r)du^{\al'}du^{\be'},
$$
and in \eqref{hbar}, all $g_{ij}$ are understood to be evaluated at
$(x,u(x,r),r)$.   

It was shown in \cite{GW} that for $g_+$ of the form \eqref{normalform} and
for $Y$ described as the graph $u^{\al'}=u^{\al'}(x,r)$, the usual minimal
submanifold equation for a graph takes the form $\cM(u)=0$, where 
\begin{equation}\label{Mform}
\begin{aligned}
\cM (u)_{\ga'} = &
\left [ r \partial_{r} - (k+1) + \frac{1}{2} r\cL_{,r} \right ] 
\left [ \hb^{00} g_{\al'\ga'}u^{\al'}{}_{,r} +
\hb^{\al 0} \left ( g_{\al\ga'} + g_{\al'\ga'}u^{\al'}{}_{,\al} 
\right ) \right ]\\
& +r\left [ \partial_{\be} + \frac{1}{2} \cL_{,\be} \right ]
\left [ \hb^{0\be} g_{\al'\ga'}u^{\al'}{}_{,r} +
\hb^{\al\be} \left ( g_{\al\ga'} + g_{\al'\ga'}u^{\al'}{}_{,\al} 
\right ) \right ] \\
& -\frac{1}{2} r \hb^{\al\be}\left [ g_{\al\be,\ga'} +
2g_{\al\al',\ga'}u^{\al'}{}_{,\be} 
+ g_{\al'\be',\ga'}u^{\al'}{}_{,\al}u^{\be'}{}_{,\be} \right ] \\
& -r \hb^{\al 0}\left [ g_{\al\al',\ga'}u^{\al'}{}_{,r} 
+ g_{\al'\be',\ga'}u^{\al'}{}_{,\al}u^{\be'}{}_{,r} \right ] \\
& -\frac{1}{2} r \hb^{00}\left [ 
g_{\al'\be',\ga'}u^{\al'}{}_{,r}u^{\be'}{}_{,r} \right ]. 
\end{aligned}
\end{equation}
Here $\cL=\log(\det \hb)$.  Components of $\hb$ and $\cM(u)$ are evaluated  
at $(x,r)$, and all $g_{ij}$ and  derivatives thereof are
evaluated at $(x,u(x,r),r)$.  The equation $\cM(u)=0$ is the equation we
will use to study the asymptotics of $U_r$.  

We next relate $\cM(u)$ to the mean curvature $H_Y$.  
Recall that $-H_Y$ is the first variation of area of $Y$, in the sense that
if $F_t:Y\rightarrow X$ is a compactly supported variation of $Y$, then 
$$
A(F_t(Y))\,{\dot{}}= -\int_Y\langle H_Y,\dot{F}\rangle_{g_+} da_Y. 
$$
Here $A$ denotes the area and $da_Y$ the area density, both with respect to 
$g_+$, and $\dot{}$ denotes $\pa_t |_{t=0}$.  The area $A(F_t(Y))$ itself is
infinite, but $A(F_t(Y))\,{\dot{}}$ is well-defined and finite since the
variation is compactly supported in $X$.  The usual derivation of the
minimal submanifold equation \eqref{Mform} for a graph amounts to
considering variations 
of the form $F_t(x,u,r)=(x,u_t,r)$ relative to coordinates $(x,u,r)$ as
above. That derivation shows that for such variations, one has   
$$
\dot{A} = -\int_Y r^{-1}\cM(u)_{\ga'}\dot{u}^{\ga'} da_Y.
$$
Therefore 
$$
r^{-1}\cM(u)_{\ga'}\dot{u}^{\ga'}=\langle H_Y,\dot{F}\rangle_{g_+}=
r^{-2}\langle H_Y,\dot{F}\rangle_{\gb}.
$$
If we write $H_Y=H^\be\pa_\be + H^{\be'}\pa_{\be'} + H^0\pa_r$, then it
follows that 
\begin{equation}\label{solveH}
g_{\be\ga'}H^\be + g_{\be'\ga'}H^{\be'} = r\cM(u)_{\ga'}.  
\end{equation}
On the other hand, $H_Y$ is normal to $Y$, so 
$$
\langle H_Y, \pa_\al + u^{\al'},_{\al} \pa_{\al'}\rangle_{\gb} = 0,\qquad
\langle H_Y, \pa_r + u^{\al'},_r \pa_{\al'}\rangle_{\gb} = 0.
$$
These can be rewritten
$$
\left(g_{\al\be}+g_{\al'\be}u^{\al'},_\al\right)H^\be
=-\left(g_{\al\be'}+g_{\al'\be'}u^{\al'},_{\al}\right)H^{\be'},
\quad
H^0=-g_{\al'\be}u^{\al'},_rH^{\be}-g_{\al'\be'}u^{\al'},_rH^{\be'}.
$$
Since $g_{\al\be}$ is smooth and nonsingular up to $r=0$ and
$g_{\al'\be}=0$  at $r=0$, the first equation can be solved to express
$H^{\be}$ as a linear function of $H^{\be'}$ near $r=0$.  The second
equation then gives $H^0$ as a function of $H^{\be'}$ near $r=0$.  Then
\eqref{solveH} can be used to solve for $H^{\be'}$ in terms of 
$r\cM(u)_{\ga'}$.  It follows that $r|\cM(u)|_{\gb}$ and $|H_Y|_{\gb}$
vanish to the same order at $r=0$.  

As discussed in \cite{GW}, the asymptotics of $u(x,r)$ can be derived
inductively from the equation $\cM(u)=0$, starting with the initial
condition $u(x,0)=0$.  For   
instance, upon simply setting $r=0$, the last four lines of 
\eqref{Mform} vanish and the first gives $u^{\al'}{}_{,r}=0$.  Suppose 
inductively that 
$u$ satisfies $\cM(u)=O(r^{m-1})$.  It is not hard to see directly
from \eqref{Mform} that 
\begin{equation}\label{inductive}
\cM(u+wr^m)_{\ga'}=\cM(u)_{\ga'}
+m(m-k-2)g_{\ga'\al'}w^{\al'}r^{m-1}+O(r^m). 
\end{equation}
(The only contribution at
order $m-1$ comes from the first term on the right-hand side.)  So if
$m<k+2$, one can uniquely determine $w|_{r=0}$ to make
$\cM(u+wr^m)=O(r^m)$.  Hence the
Taylor expansion of $u \mod O(r^{k+2})$ is uniquely 
determined by the equation $|\cM(u)|_{\gb}=O(r^{k+1})$.  By the discussion
in the previous paragraph, this corresponds to $|H_Y|_{\gb}=O(r^{k+2})$.
That $U_r$ is even follows by inspection of \eqref{Mform}:  the map 
$u\mapsto \cM(u)$ respects parity.  

To analyze what happens at order $k+2$, let $v$ be smooth and satisfy
$\cM(v)=O(r^{k+1})$.  It follows from \eqref{inductive} with $m=k+2$ that 
$r^{-k-1}\cM(v)|_{r=0}$ is independent of the choice of the order $r^{k+2}$
ambiguity in $v$, and if nonzero, is therefore an obstruction to solving
$\cM(u)=O(r^{k+2})$ with $u$ smooth.  To solve at this order it is
necessary to introduce a log term.  One calculates from \eqref{Mform} that 
$$
\cM(v+wr^{k+2}\log r)_{\ga'}
=\cM(v)_{\ga'}+(k+2)g_{\ga'\al'}w^{\al'}r^{k+1}+O(r^{k+2}|\log
r|). 
$$
Hence $w^{\al'}=-(k+2)^{-1}g^{\al'\ga'}r^{-k-1}\cM(v)_{\ga'}|_{r=0}$ is the
unique choice to make  
$u=v+wr^{k+2}\log r$ satisfy $\cM(u)=O(r^{k+2}|\log r|)$.  

Since $v$ and $u$ are unique mod $O(r^{k+2})$ and the equation $\cM(u)=0$
is a coordinate representation of the invariant condition $H_Y=0$, the 
corresponding 1-parameter families $V_r=v^{\al'}e_{\al'}$ and  
$U_r=u^{\al'}e_{\al'}$ of sections of $N\Si$ are globally and invariantly 
defined mod $O(r^{k+2})$.  
The normal space to $Y$ at $r=0$ is $\Span
\{\pa_{\be'}\}$, and it follows from \eqref{solveH} that
$g_{\ga'\be'}r^{-k-2}H^{\be'}|_{r=0}=r^{-k-1}\cM(v)_{\ga'}|_{r=0}$.  Hence 
the definition $\cH:=r^{-k-2}H_Y|_{r=0}$ is equivalent to: 
\begin{equation}\label{cH}
\cH^{\al'}=g^{\al'\ga'}r^{-k-1}\cM(v)_{\ga'}|_{r=0}.  
\end{equation}
The determination of $w$ above therefore shows that the coefficient of  
$r^{k+2}\log r$ in $U_r$ is $-(k+2)^{-1}\cH$. 
\end{proof}

We write the expansion of $U_r$ in the form
\begin{equation}\label{Uexpand}
U_r = U_{(2)}r^2+\ldots +U_{(k)}r^k
-(k+2)^{-1}\cH r^{k+2} \log r + \ldots,
\end{equation}
where each $U_{(2j)}$, $1\leq j\leq k/2$, is a globally, invariantly
defined section of $N\Sigma$ determined by the choice of metric $g$ in the
conformal class.  The $U_{(2j)}$ are not conformally invariant, but $\cH$
is: 

\begin{proposition}\label{invL}
If $\gh=\Omega^2g$ with $\Omega\in C^\infty(M)$, then
$\cHh=(\Omega|_{\Sigma})^{-(k+2)}\cH$. 
\end{proposition}
\begin{proof}
Write $\gh_+=\rh^{-2}(d\rh^2 + \gh_{\rh})$ for the analogue of
\eqref{normalform} with respect to $\gh$.  Then there is a
diffeomorphism  
$\psi$ on a neighborhood of $M$ in $M\times [0,\ep_0)$, restricting to the  
identity on $M\times \{0\}$, for which $\psi^*\gh_+=g_+\mod O(r^{n-2})$
and $\psi^*\rh=\Omega r+O(r^2)$ (\cite{FG}).  If $Y$ satisfies    
$|H_Y|_{\gb}=O(r^{k+2})$, then $\Yh=\psi(Y)$ satisfies 
$|H_{\Yh}|_{\overline{\widehat{g}}}=O(\rh^{k+2})$, 
where the mean curvature of $\Yh$ is taken with respect to $\gh_+$.  Since 
$\psi$ restricts to the identity on $M$, it follows that  
$$
\cHh = \psi^*\cHh=\psi^*\left(\rh^{-k-2}H_{\Yh}|_{\rh=0}\right)
=\Omega^{-(k+2)} r^{-k-2}H_Y|_{r=0} =(\Omega|_{\Sigma})^{-(k+2)}\cH. 
$$
\end{proof}

\section{Energy}\label{energysection}

In this section we consider immersed submanifolds of $M$.  Thus let $\Si$
be a manifold of even dimension $k$ and $f:\Si\rightarrow M$ an immersion.  
Relative to the metric $g$ on $M$, the pullback bundle $f^*TM$ splits as  
$$
f^*TM=T\Si\oplus N\Si.
$$
$f$ is locally an embedding, so the considerations of the previous section
apply.  In particular, Theorem~\ref{U} determines a 1-parameter family
of sections $U_r\mod O(r^{k+2})$ of $N\Si$ and an obstruction field $\cH\in 
\Gamma(N\Si)$.  In this section we set $Y=\Si\times [0,\ep_0)$ immersed in 
$\Xb=M\times [0,\ep_0)$ via the map 
$\ft:\Si\times [0,\ep_0)\rightarrow \Xb$ given by: 
\begin{equation}\label{Phi}
\ft(p,r)=\big(\exp_\Si U_r(p),r\big). 
\end{equation}

Consider the asymptotics of the area density $da_Y$ for the metric induced  
by $g_+$.  We have $da_Y=\varphi\,da_\Sigma dr$ for
an invariantly defined function $\varphi$ on $\Sigma\times (0,\ep_0)$.
Here $da_\Si$ denotes the area density of $\Si$ with respect to the metric 
induced by $g$.  In terms of local coordinates $(x^\al,u^{\al'})$
introduced in the proof of Theorem~\ref{U}, 
we have
$$
\varphi(x,r)=\sqrt{\frac{\det h(x,r)}{\det \hb_{\al\be}(x,0)}} 
=r^{-k-1}\sqrt{\frac{\det \hb(x,r)}{\det \hb_{\al\be}(x,0)}}
$$
with $\hb$ given by \eqref{hbar}.  
Since $U_r$ is even in $r$ to order $k+2$, it follows that the expansion of
$\sqrt{\frac{\det \hb(x,r)}{\det \hb_{\al\be}(x,0)}}$ has only even terms 
through order $k$.  Hence we can write
\begin{equation}\label{invariantexpand}
da_Y=
r^{-k-1}\left[a^{(0)}+a^{(2)}r^2+\ldots
  +a^{(k)}r^k+\ldots\right]\,da_\Sigma dr
\end{equation}
for invariantly defined functions $a^{(2j)}$, $1\leq j\leq k/2$, on
$\Sigma$ determined by    
\begin{equation}\label{areaexpand}
\sqrt{\frac{\det \hb(\cdot,r)}{\det \hb_{\al\be}(\cdot,0)}}   
= a^{(0)}+a^{(2)}r^2+\ldots +a^{(k)}r^k+\ldots. 
\end{equation}
In particular, $a^{(0)}= 1$.  The $a^{(2j)}$ are called the renormalized
area coefficients for $\Si$.  

Assume now that $\Si$ is compact.  It follows
upon integration of \eqref{invariantexpand} that for $\epsilon_0$ fixed,  
\begin{equation}\label{area}
\operatorname{Area}(Y\cap \{\ep<r<\ep_0\}) =
A_0\ep^{-k}+A_2\ep^{-k+2}+\ldots + 
A_{k-2}\ep^{-2} + \cE \log \frac{1}{\ep} +O(1)
\end{equation}
as $\ep\rightarrow 0$, with
\begin{equation}\label{areacoeffs}
A_{2j}=\frac{1}{k-2j}\int_\Sigma a^{(2j)}\,da_\Sigma,\quad 0\leq j\leq k/2-1,
\qquad \cE=\int_\Sigma a^{(k)}\,da_\Sigma. 
\end{equation}

\begin{proposition}\label{invariance}
$\cE$ is independent of the choice of representative metric $g$.
\end{proposition}
\begin{proof}
Let $\gh$ be a conformally related metric.  There is a uniquely determined
defining function 
$\rh$ in a neighborhood of $M$ in $\Xb$ such that $\rh^2g_+|_{TM}=\gh$ and  
$|d\rh/\rh|_{g_+}=1$.  The difference $\cE-\widehat{\cE}$ is the
coefficient of $\log\frac{1}{\ep}$ in the expansion of 
$\operatorname{Area}(Y\cap \{\ep<r\}) -\operatorname{Area}(Y\cap
\{\ep<\rh\})$.  Now $\rh>\ep$ is equivalent to $r>\ep b(x,\ep)$ for 
a smooth positive function $b(x,\ep)$.  Writing 
$\operatorname{Area}(Y\cap \{\ep<r\}) -\operatorname{Area}(Y\cap
\{\ep<\rh\})$ as an integral, it follows without difficulty that
the coefficient of $\log\frac{1}{\ep}$ in its expansion is equal to zero. 
See Proposition 2.1 of \cite{GW} for details.    
\end{proof}

The motivation for viewing $\cE$ as a version of the Willmore energy is
the fact, derived in \cite{GW}, that when $k=2$ and $(M,g)$ is
3-dimensional Euclidean space, $\cE$ reduces to a multiple of the usual 
Willmore energy of $\Si$.  This derivation will be reviewed in
\S\ref{formulas}.  

\begin{remark} There are other natural invariant immersions 
$\Si\times [0,\ep_0)\rightarrow M\times [0,\ep_0)$ having the same image as  
$\ft$, which give rise to different  
coefficients $a^{(2j)}$ in \eqref{invariantexpand}.  The quantity
$\operatorname{Area}(Y\cap \{\ep<r<\ep_0\})$ is independent of the choice 
of parametrization, so as long as the immersion takes the form 
$(p,r)\rightarrow (\Phi_r(p),r)$ for a 1-parameter family of
immersions $\Phi_r:\Si\rightarrow M$ satisfying $\Phi_0=f$,
the coefficients $A_{2j}$ given by \eqref{areacoeffs} will be the same.
The corresponding $a^{(2j)}$ will differ by a divergence.   Consequently, one 
might not expect that the specific coefficients $a^{(2j)}$ will play as 
fundamental a role as they do for the case of volume renormalization, where 
there is a canonical parametrization.    
\end{remark}

Next we consider the variational derivative of $\cE$ on the space of
immersions of $\Si$ into $M$.  
Let $F_t:\Sigma \rightarrow M$, $0\leq t< \delta$ be a variation of  
$\Sigma$, i.e. a smooth 1-parameter family of immersions with $F_0=f$.
Denote by $F:\Si\times [0,\de)\rightarrow M$ the map $F(p,t)=F_t(p)$.  
Let 
$\Sigma_t$ denote $\Si$ immersed into $M$ via $F_t$, let $a^{(k)}_t$ 
be the corresponding renormalized volume coefficient, and set $\cE_t = 
\int_{\Sigma_t}a^{(k)}_tda_{\Sigma_t}$.  
Also set $\dot {F} =\pa_t F|_{t=0} 
\in \Gamma(f^*TM)$ and $\dot{\cE}=\pa_t \cE_t|_{t=0}$.  
\begin{theorem}\label{variation}
If $k\geq 2$, then              
$$
\dot{\cE} = - \int_\Sigma \langle\dot{F},\cH \rangle_g \,da_\Sigma . 
$$
\end{theorem}
\begin{proof}

For each $t$, let $U_r^t$ be the 1-parameter family of sections of
$N\Si_t$ determined by $\Sigma_t$ modulo $O(r^{k+2})$ as in 
Theorem~\ref{U}.  For definiteness, we fix the indeterminacy in $U_r^t$
by truncating the expansion \eqref{Uexpand} after the log term: 
\begin{equation}\label{Utexpand}
U_r^t(p)= U_{(2)}^t(p)r^2 + \ldots + U_{(k)}^t(p)r^k -(k+2)^{-1}\cH^t(p)
r^{k+2} \log r,\qquad p\in \Si.  
\end{equation}
Let $\Ft_t:\Si\times [0,\ep_0)\rightarrow \Xb$ be the immersion defined
by analogy to \eqref{Phi}:
\begin{equation}\label{Ft}
\Ft_t(p,r)
=\big(\exp_{\Si_t} U_r^t(p),r\big),
\end{equation}
and denote by $Y_t$ be the corresponding immersed submanifold of $\Xb$.  
Then $Y_0=Y$ and $U^0_r=U_r$.  

Fix $\ep_0$ small and let $0<\ep<\ep_0$.   Set 
$Y_t^\ep=Y_t\cap\{\ep<r<\ep_0\}$ and $Y^\ep = Y^\ep_0$.  Then 
$\Ft_t|_{\Si\times (\ep,\ep_0)}:\Si\times (\ep,\ep_0)\rightarrow Y_t^\ep$
is a variation of  the manifold-with-boundary $Y^\ep$.   
The first variation of area formula for $\Ft_t|_{\Si\times (\ep,\ep_0)}$  
with background metric $g_+$ states    
\begin{equation}\label{Area}
A(Y_t^\ep)\,\dot{} 
= -\int_{Y^\ep}\langle H_Y,\dot{\Ft}\rangle_{g_+}\,da_Y
+\left(\int_{Y\cap \{r=\ep_0\}} + \int_{Y\cap \{r=\ep\}}\right)  
\langle n,\dot{\Ft}\rangle_{g_+} da_\pa,
\end{equation}
where $da_\pa$ denotes the induced area density and $n$ the outward 
pointing normal on $\pa Y^\ep=(Y\cap \{r=\ep_0\})\cup (Y\cap \{r=\ep\})$.    
Both sides of this equation blow up as $\ep\rightarrow 0$.  We consider
their asymptotic expansions in $\ep$.  

According to \eqref{area}, we have 
$$
A(Y^\ep_t)\,\dot{} = \dot{A}_0\ep^{-k}+\dot{A}_2\ep^{-k+2}+\ldots +
\dot{A}_{k-2}\ep^{-2} + \dot{\cE} \log \frac{1}{\ep} +O(1).
$$
So $\dot{\cE}$ occurs as the coefficient of $\log \frac{1}{\ep}$ in the
asymptotic expansion of the left-hand side of \eqref{Area}.  The proof will
be concluded by showing that the coefficient of $\log \frac{1}{\ep}$ on the
right-hand side is $-\int_\Sigma \langle\dot{F},\cH \rangle_g
\,da_\Sigma$.  It suffices to assume that $\dot{F}$ is supported in a small
open set in $\Si$.  In the following argument, we sometimes reduce $\ep_0$
and $\de$ without mention.       

We begin by analyzing $\dot{\Ft}$, a section of $\ft^*T\Xb$.  
Certainly $\dot{\Ft}|_\Si = \dot{F}$.  The  
decomposition $\Xb= M\times [0,\ep_0)$ induces a decomposition   
$T\Xb= TM\oplus T([0,\ep_0))$.  It is clear from \eqref{Ft} that the  
$T([0,\ep_0))$-component 
of $\dot{\Ft}$ vanishes at each point.  Choose $\cV\subset \Si$ open and  
local coordinates $z=(z^1,\ldots,z^n)$ for $M$ in a
neighborhood of $f(\cV)$.  We can write 
\begin{equation}\label{Fcoef}
\dot{\Ft} = \dot{\Ft}{}^i \pa_{z^i} 
\end{equation}
with coefficients $\dot{\Ft}{}^i$ which are 
functions on $\cV\times [0,\ep_0)$.
\begin{lemma}\label{Fdot}
Each $\dot{\Ft}{}^i$ has an asymptotic expansion of the form 
\begin{equation}\label{expand}
f_0 + f_2r^2 + \ldots + f_kr^k 
+ f_{\log}r^{k+2}\log r +O(r^{k+2})
\end{equation}
with coefficients $f_{2j}, f_{\log}\in C^\infty(\cV)$.    
\end{lemma}
\begin{proof}
As bundles on $\Si\times [0,\de)$, the pullback bundle splits as 
$F^*TM=T\oplus N$, where $T_{(p,t)}=T_p\Si_t$, $N_{(p,t)}=N_p\Si_t$.  We
define $\exp:N\rightarrow M$ near the zero section by
$\exp v=\exp_{\Si_t}v$ for $v\in N_{(p,t)}=N_p\Si_t$.  Let  
$e_{\al'}(p,t)$, $1\leq \al'\leq n-k$, be a smooth frame for  
$N|_{\cV\times [0,\de)}$.  
This frame determines a diffeomorphism
$\chi:\cV\times [0,\de)\times \cU\rightarrow N$ onto its image, for $\cU$ a 
neighborhood of the origin in $\R^{n-k}$, by
$$
\chi(p,t,u)=\big(p,t,u^{\al'}e_{\al'}(p,t)\big).
$$
If we represent points of $M$ using the local coordinates $z=(z^1,\ldots, 
z^n)$, then in these coordinates the map $\exp$ can be expressed as 
$$
(\exp\circ\chi)(p,t,u)=z(p,t,u),
$$
where $z(p,t,u)$ is a smooth $\R^n$-valued function on $\cV\times
[0,\de)\times \cU$.  
In these terms, the definition \eqref{Ft} of $\Ft_t$ becomes
$$
(\Ft_t)(p,r)=\left(z\big(p,t,u_r(p,t)\big),r\right),
$$
where $u_r:\cV\times [0,\de)\rightarrow \cU$ denotes the components of 
$U_r^t$ in the 
frame $e_{\al'}$, defined by  
\begin{equation}\label{usubr}
U_r^t(p)=u^{\al'}_r(p,t)e_{\al'}(p,t).
\end{equation}

Now $z$ is a smooth function of $(p,t,u)$.  
So the asymptotic expansion in $r$ of the $z$-components of $\Ft_t$
can be obtained by composing the Taylor expansion of $z$ about $u=0$ with 
the expansion of the $u^{\al'}_r$, which are determined by combining
\eqref{Utexpand} with \eqref{usubr}.   
It follows that each $z$-component of $\Ft_t$ has an expansion of the 
form \eqref{expand} with coefficients depending smoothly on $t$. 
Differentiation in $t$ at $t=0$ yields the stated claim concerning 
$\dot{\Ft}$.    
\end{proof}

Return now to consider the right-hand side of \eqref{Area}.  
According to Theorem~\ref{U}, we have 
$|H_Y|_{\gb}=O(r^{k+3}|\log r|)$.    
Lemma~\ref{Fdot} shows that $|\dot{\Ft}|_{\gb}=O(1)$.  Consequently
$|\langle H_Y,\dot{\Ft}\rangle_{g_+}| =O(r^{k+1}|\log r|)$.  Since 
$da_Y=O(r^{-k-1})da_\Sigma dr$, we deduce that
$$
\Big|\int_{Y^\ep}\langle H_Y,\dot{\Ft}\rangle_{g_+}\,da_Y\Big|=O(1) 
$$
as $\ep\rightarrow 0$.  In particular, this term does not contribute to the
$\log\frac{1}{\ep}$ term in the expansion of the right-hand side of
\eqref{Area}.  Likewise, the integral over $Y\cap \{r=\ep_0\}$ is
independent of $\ep$, so does not contribute to the $\log\frac{1}{\ep}$
term.  So the $\log \frac{1}{\ep}$ term in the asymptotic expansion of the
right-hand side of \eqref{Area} equals that for 
$\int_{Y\cap\{r=\ep\}}\langle n,\dot{\Ft}\rangle_{g_+}\,da_\ep$.  Here 
we denote the induced area density on $Y\cap\{r=\ep\}$ by $da_\ep$.     
We regard $Y\cap\{r=\ep\}$ as the immersed submanifold of $X$ defined by  
the immersion 
$\ft_\ep:\Sigma\rightarrow X$,  $\ft_\ep(p)=\ft(p,\ep)$.    
We study the pointwise asymptotics in $\ep$ of 
$\langle n,\dot{\Ft}\rangle_{g_+}|_{r=\ep}$ and of $da_\ep$ using local   
coordinates.    

Choose a local coordinate system $\{x^\al: 1\leq 
\al\leq k\}$ for an open subset $\cV\subset \Sigma$ and a local frame
$\{e_{\al'}(x): 1\leq \al'\leq  n-k\}$ for $N\Sigma|_{\cV}$ as in the proof
of Theorem~\ref{U}.  In the corresponding local coordinates $(x,u)$ 
for $M$ near $f(\cV)$, $Y$ is given by $u^{\al'}=u^{\al'}(x,r)$, where
$u^{\al'}(x,r)$ are the components of $U_r(x)$ in the frame
$\{e_{\al'}(x)\}$. 
The metric on $Y$ induced by $g_+$ takes the 
form $h=r^{-2}\hb$ in the local coordinates $(x,r)$ on $\Sigma\times
[0,\ep_0)$, with $\hb$ given by \eqref{hbar}.  The outward unit 
conormal to $\{r=\ep\}$ is $-dr/|dr|_h$, so the outward unit normal is
given in these $(x,r)$ coordinates by 
$$
-\frac{1}{\sqrt{h^{00}}}\left(h^{0\al}\,\pa_\al + h^{00}\,\pa_r\right)=
-\left(\frac{h^{0\al}}{\sqrt{h^{00}}}\,\pa_\al + \sqrt{h^{00}}\,\pa_r\right).  
$$
Thus
\[
\begin{split}
-n&=\ft_*\left(\frac{h^{0\al}}{\sqrt{h^{00}}}\,\pa_\al +
\sqrt{h^{00}}\,\pa_r\right)
=\frac{h^{0\al}}{\sqrt{h^{00}}}\left(\pa_\al +
u^{\al'}{}_{,\al}\,\pa_{\al'}\right)
+\sqrt{h^{00}}\left(\pa_r +
u^{\al'}{}_{,r}\pa_{\al'}\right)\\
&=\frac{h^{0\al}}{\sqrt{h^{00}}}\pa_\al 
+\left(\frac{h^{0\ga}}{\sqrt{h^{00}}}u^{\al'}{}_{,\ga}
+\sqrt{h^{00}}u^{\al'}{}_{,r}\right)\pa_{\al'}  
+\sqrt{h^{00}}\,\pa_r.
\end{split}
\]
All metric coefficients $h^{0\al}$, $h^{00}$ are evaluated at $(x,\ep)$.

Take the coordinates $z$ in \eqref{Fcoef} to be $z=(x,u)$.  Then
\eqref{Fcoef} becomes
$$
\dot{\Ft} = \dot{\Ft}{}^{\al}\pa_\al +\dot{\Ft}{}^{\al'}\pa_{\al'}.  
$$
Recalling \eqref{normalform}, it follows that 
\begin{equation}\label{pb1}
\begin{split}
-\ep\langle n,\dot{\Ft} \rangle_{g_+}
=&\,\,g_{\al\be}\dot{\Ft}{}^{\al}\frac{\hb{}^{0\be}}{\sqrt{\hb{}^{00}}}
+g_{\al\be'}\dot{\Ft}{}^{\al}\left(\frac{\hb{}^{0\ga}}{\sqrt{\hb{}^{00}}}u^{\be'}{}_{,\ga}
+\sqrt{\hb{}^{00}}u^{\be'}{}_{,r}\right)\\
&+g_{\be\al'}\dot{\Ft}{}^{\al'}\frac{\hb{}^{0\be}}{\sqrt{\hb{}^{00}}}
+g_{\al'\be'}\dot{\Ft}{}^{\al'}\left(\frac{\hb{}^{0\ga}}{\sqrt{\hb{}^{00}}}u^{\be'}{}_{,\ga}
+\sqrt{\hb{}^{00}}u^{\be'}{}_{,r}\right).
\end{split}
\end{equation}
All $g_{ij}$ are evaluated at $(x,u(x,\ep),\ep)$.  
Likewise,
\begin{equation}\label{pb2}
da_\ep=\sqrt{\det h_{\al\be}(x,\ep)}\,dx
=\ep^{-k}\sqrt{\det \hb_{\al\be}(x,\ep)}\,dx
=\ep^{-k}\sqrt{\frac{\det \hb_{\al\be}(x,\ep)}{\det\hb_{\al\be}(x,0)}}
\,da_\Sigma.
\end{equation}

Consider the asymptotic expansion in $\ep$ of each of the terms appearing
in the right-hand sides of \eqref{pb1}, \eqref{pb2}.  In \eqref{pb1}, 
the factors $g_{ij}$, $\dot{\Ft}{}^i$, $\hb{}^{0\be}$, 
$\sqrt{\hb{}^{00}}$, $1/\sqrt{\hb{}^{00}}$, $u^{\be'}{}_{,\ga}$ and 
$u^{\be'}{}_r$ all have expansions in 
nonnegative powers of $\ep$ and positive powers of $\log\ep$.  We analyze
the powers of $\ep$ multiplying the $\log\ep$ terms in the expansions.  

First consider  
$g_{ij}(x,u(x,\ep),\ep)$.  Since $g_{ij}(x,u,r)$ is
smooth, the asymptotic expansion of $g_{ij}(x,u(x,\ep),\ep)$ is obtained by
composing the Taylor expansion of $g_{ij}$ in $u$ and $r$ about $u=0$,
$r=0$ with the asymptotic expansion of $u$ in $r$, and then setting
$r=\ep$.  Thus it follows 
from \eqref{Utexpand} with $t=0$ that each $\log \ep$ term in   
$g_{ij}(x,u(x,\ep),\ep)$ occurs muliplied by $\ep$ to a power at least
$k+2$.  Next consider the induced metric coefficients $\hb_{\al\be}$,  
$\hb_{\al 0}$, $\hb_{00}$ given by \eqref{hbar}.  We claim likewise that
$\log \ep$ occurs in each of these muliplied by $\ep$ to a power at least  
$k+2$.  For $\hb_{\al\be}$ this is clear since the derivatives of
$u^{\al'}$ which appear are tangential to $\Sigma$.  Now $u^{\al'}{}_{,r}$
has a term of the form $\ep^{k+1}\log \ep$.  However, since $g_{\al\al'}$, 
$u^{\al'}{}_{,\al}$ and $u^{\al'}{}_{,r}$ all vanish at $\ep=0$, the
log terms in the expansions of the $u^{\al'}{}_{,r}$ occurring 
in $\hb_{\al 0}$ and $\hb_{00}$ all get multiplied by at least one extra
factor of $\ep$, and the claim 
follows.  We conclude that each $\log \ep$ in the inverse
metric coefficients $\hb^{\al\be}$, $\hb^{\al 0}$, $\hb^{00}$ also 
is muliplied by $\ep$ to a power at least $k+2$.  Lemma~\ref{Fdot} shows
that each $\log \ep$ in $\dot{\Ft}{}^{\al}$ and 
$\dot{\Ft}{}^{\al'}$ is muliplied by $\ep$ to a power at least
$k+2$.  

The first term on the right-hand side of \eqref{pb1} is 
$g_{\al\be}\dot{\Ft}{}^{\al}\hb{}^{0\be}/\sqrt{\hb{}^{00}}$.
{From} the above observations it is clear that each $\log \ep$ term in its
asymptotic expansion is muliplied by $\ep$ to a power at least $k+2$.
Likewise for the third term 
$g_{\be\al'}\dot{\Ft}{}^{\al'}\hb{}^{0\be}/\sqrt{\hb{}^{00}}$. 
The second and fourth terms of \eqref{pb1} include a factor
$u^{\be'}{}_{,r}$, whose expansion has a term of the form
$\ep^{k+1}\log \ep$.  Now the second term has a 
leading factor $g_{\al\be'}$, which vanishes at $\ep=0$.  So 
each $\log \ep$ term in the asymptotic expansion of the second term is
muliplied by $\ep$ to a power at least $k+2$.  However, this is not the
case for the fourth term.   According to \eqref{Utexpand}, $u^{\be'}{}_{,r}$
has a term $-\cH^{\be'}\ep^{k+1}\log\ep$.  Since 
$\dot{\Ft}|_\Sigma  =\dot{F}$ and $\hb^{00}|_\Sigma=1$, it follows that the
expansion of the fourth term of \eqref{pb1} has a term 
$-g_{\al'\be'}\dot{F}^{\al'}\cH^{\be'}\,\ep^{k+1}\log\ep$.  Putting all this
together, we conclude that the  
expansion of $\langle n,\dot{\Ft} \rangle_{g_+}$
has a term $\langle \dot{F},\cH\rangle_g\,\ep^k\log\ep$, 
and all other $\log\ep$ terms appear with a coefficient of
$\ep^{k+1}$ or higher.     

It is clear that the factor 
$\sqrt{\det \hb_{\al\be}(x,\ep)/\det\hb_{\al\be}(x,0)}$ in \eqref{pb2} has
an expansion 
with a leading term of $1$ and with all $\log\ep$ terms multiplied by a
power of $\ep$ at least $k+2$.  Combining with the conclusion of the above
paragraph, it follows that the $\log\ep$ coefficient in the expansion of  
$\langle n,\dot{\Ft}\rangle_{g_+}da_\ep$ is 
$\langle \dot{F},\cH\rangle_g\,\,da_\Sigma$.
Integrating over $\Sigma$ concludes the proof of Theorem~\ref{variation}.    
\end{proof}

As a consequence of Theorem~\ref{variation}, we deduce the following
proposition. 
\begin{proposition}\label{minimalcritical}
Let $k\geq 2$ be even and suppose $(M^n,g)$ is Einstein.  If $\Si^k$ is a
minimal immersed submanifold of $(M,g)$, then $\Si$ is critical for $\cE$.    
\end{proposition}
\begin{proof}
The Poincar\'e metric for an Einstein metric can be written explicitly:
if $\Ric(g)=2\lambda(n-1)g$, then $g_r=(1-\lambda r^2/2)^2g$ (see
\cite{FG}).  So 
\[
g_+=r^{-2}\big(dr^2+(1-\lambda r^2/2)^2g\big)=ds^2+(e^s-\lambda
e^{-s}/2)^2g,\qquad s=-\log r. 
\]
It is easy to verify the general fact that if $\Si$ is a minimal
submanifold of a Riemannian manifold $(M,g)$, 
then $\Si\times \R$ is a minimal submanifold of $M\times \R$ with
respect to any warped product metric of the form $g_+=ds^2+A(s)g$, where
$s$ denotes the variable in $\R$ and $A(s)$ is a positive function.  Thus 
when $g$ is Einstein and $\Si$ is minimal, the minimal extension $Y$ in  
Theorem~\ref{U} is simply $Y=\Si\times \R$.  That is, the  
corresponding normal field is $U_r=0 \mod O(r^{k+2})$.  (An alternate,
equivalent way to see this is simply to note that if $g_r=B(r)g$ for some
positive function $B(r)$ and $\Si$ is minimal, then $u=0$ solves $\cM(u)=0$
exactly, where recall $\cM(u)$ is given by \eqref{Mform}).  

Since $U_r=0$ has no log term in its expansion, it must be that $\cH=0$.
By Theorem~\ref{variation}, it follows that $\Si$ is critical for $\cE$.   
\end{proof}

\noindent
Proposition~\ref{minimalcritical} implies in particular that minimal 
submanifolds of $\R^n$ or $S^n$ are critical for $\cE$.  

\section{Derivation of Formulas}\label{formulas}

In this section we derive formulas for the renormalized area coefficients
$a^{(2)}$, $a^{(4)}$ in the expansion \eqref{areaexpand}, and for the
coefficients $U_{(2)}$, $U_{(4)}$ in the expansion \eqref{Uexpand} of $U_r$.  This gives 
formulas for the energy $\cE$ for $k=2$, $4$ by integration, and for the 
obstruction $\cH$ for $k=2$.  We also use Theorem~\ref{variation} and a 
formalism of Guven to identify $\cH$ for $k=4$ when $(M,g)$ is a Euclidean 
space.  

The coefficients $U_{(2j)}$ are determined by solving the equation
$\cM(u)=0$ inductively order by order, where $\cM(u)$ is given by
\eqref{Mform}.   
Using the fact that $u^{\al'}{}_{,r}=0$ at $r=0$, one sees easily that all
terms on the right-hand side of \eqref{Mform} are $O(r^2)$ except for 
the first and third.  Thus
$$
(r\pa_r-(k+1))(\hb^{00}g_{\al'\ga'}u^{\al'}{}_{,r})
-\tfrac12 r\hb^{\al\be}g_{\al\be.\ga'}=O(r^2).
$$
Applying $\pa_r|_{r=0}$ and raising an index gives
$ku^{\al'}{}_{,rr} =-\frac12 g^{\al'\ga'}g^{\al\be}g_{\al\be,\ga'}$ at
$r=0$.  Therefore
\begin{equation}\label{U2}
U_{(2)}=\frac{1}{2k}H.
\end{equation}

We next turn to the identification of $a^{(2)}$ and $a^{(4)}$.  We will
return later to the determination of $U_{(4)}$ by further differentiation
of \eqref{Mform}.   

\begin{proposition}\label{a's}
If $k\geq 1$, then 
\[
\begin{split}
 a^{(2)}=-\frac12 & \left(\frac{k-1}{k^2}|H|^2+P^\al{}_\al\right)\\  
a^{(4)} =\frac{1}{8k^2} & \Bigg( |\nabla H|^2 -
L^{\al'}_{\al\be}L^{\al\be}_{\be'}H_{\al'}H^{\be'}
+\frac{k^2-2k-1}{k^2} |H|^4\\
& -W^\al{}_{\al'\al\be'}H^{\al'}H^{\be'}  
-2k g^{\al\be}{}^M\nabla_{\al'}P_{\al\be} H^{\al'}   
-4k P^\al{}_{\al'}\nabla_\al H^{\al'}\\ 
&+(2k-3)P^\al{}_\al|H|^2 -(k+4)P_{\al'\be'}H^{\al'}H^{\be'} \Bigg) \\
+&\frac18 \Big( -P^{\al\be}P_{\al\be}+ P^{\al\al'}P_{\al\al'}+(P^\al{}_\al)^2 
-\frac{1}{n-4}B^\al{}_\al\Big)\\
+&\frac{4-k}{k}H_{\al'}U^{\al'}_{(4)}
\end{split}
\]
\end{proposition}
In the last line of the expression for $a^{(4)}$, one should substitute the
formula for 
$U_{(4)}$ given in Proposition~\ref{U4prop} and combine like terms.
However, we are primarily interested in $a^{(4)}$ for $k=4$, when it  
is the integrand for $\cE$.  When $k=4$, the $U_{(4)}$ term does not
appear.  Consequently we have left the expression in the above form.   

Proposition~\ref{a's} will be proved by calculating in special coordinates.   
Recall the geodesic normal coordinate systems $(x,u)$ on $M$ near $\Si$
constructed at the beginning of the proof of Theorem~\ref{U} which are  
associated to a   
choice of local coordinates $x^\al$ for $\Sigma$ and a choice of frame
$e_{\al'}$ for $N\Si$.  Given $p\in \Si$, choose the $x^{\al}$ so that
$g_{\al\be,\ga}(p)=0$ for $1\leq \al,\be,\ga \leq k$.  (For instance,
take $x^\al$ to be geodesic normal coordinates at $p$ for the
induced metric on $\Si$.)  Choose the frame $e_{\al'}$ so that 
$\nabla_\al e_{\al'}(p)=0$ for $1\leq \al \leq k$, $1\leq \al'\leq n-k$,
where $\nabla_\al$ denotes the normal bundle connection on $N\Si$.  
\begin{lemma}\label{derivsvanish}
In such coordinates $(x,u)$, the following all vanish at $p$: 
$$
g_{\al\be,\ga}\qquad g_{\al\al',\be}\qquad g_{\al'\be',\al}
\qquad g_{\al\al',\be'} \qquad g_{\al'\be',\ga'}.
$$
At $p$ we also have 
\begin{equation}\label{whichone}
g_{\al\be,\al'}= -2 L_{\al\be\al'}
\end{equation}
and
\begin{equation}\label{R}
g_{\al\be,\al'\be'}=2R_{\al'(\al\be)\be'}+L_{\al\ga\be'}L_{\be}{}^\ga{}_{\al'} 
+L_{\be\ga\be'}L_{\al}{}^\ga{}_{\al'}.
\end{equation}
If $v$ is a section of $N\Si$, then at $p$ we have
\begin{equation}\label{hess}
\na_\al\na_\be
v^{\al'}=\big(v^{\al'}{}_{,\be\al}+g^{\al'\ga'}g_{\be\ga',\be'\al}v^{\be'}\big).  
\end{equation}
\end{lemma}
\begin{proof}
The $x$ coordinates were chosen so that $g_{\al\be,\ga}=0$ at $p$.  We have  
$g_{\al\al',\be}=0$ on all of $\Si$ since $g_{\al\al'}=0$ on $\Si$ and
$\pa_\be$ acts tangentially.  Recall from the construction of
$(x,u)$ that each curve $t\mapsto (x,tu)$ is a geodesic.  This implies in
particular that 
$\Ga^{\ga}_{\al'\be'}=0$ and $\Ga^{\ga'}_{\al'\be'}=0$ on $\Si$.  The  
latter equation is equivalent to $g_{\al'\be',\ga'}=0$ on $\Si$.  The
former is equivalent to
\begin{equation}\label{gderivs}
g_{\al'\be',\al}=g_{\al\al',\be'}+g_{\al\be',\al'}\qquad \text{on } \Si. 
\end{equation}
Since $e_{\al'}=\pa_{\al'}$ on $\Si$, the equation $\nabla_\al
e_{\al'}(p)=0$ is equivalent to 
$\Gamma_{\al\al'}^{\be'}(p)=0$, which is equivalent to
$g_{\al'\be',\al}=g_{\al\al',\be'}-g_{\al\be',\al'}$ at $p$. 
The
left-hand side is symmetric in $\al'\be'$ and the right-hand side is skew,
so both must vanish.  
Combining with \eqref{gderivs}, we conclude that in fact 
$g_{\al'\be',\al}=g_{\al\al',\be'}=0$ at $p$.   

Equation \eqref{whichone} holds on all of $\Si$ in any coordinates $(x,u)$
for which 
$\Si=\{u=0\}$ and $\pa_\al\perp\pa_{\al'}$ on $\Si$.  In fact, in this case 
$\na_{\pa_\al}\pa_{\be}=\Ga_{\al\be}^{\ga}\pa_\ga
+\Ga_{\al\be}^{\ga'}\pa_{\ga'}$, so 
$L_{\al\be}^{\ga'}=\Ga_{\al\be}^{\ga'}=-\frac12
g^{\ga'\al'}g_{\al\be,\al'}$.  

For \eqref{R}, the curvature tensor is given in 
local coordinates by
$$
-2R_{ijkl}=g_{ik,jl}+g_{jl,ik}-g_{il,jk}-g_{jk,il}
-2g^{pq}(\Ga_{ilp}\Ga_{jkq}-\Ga_{ikp}\Ga_{jlq}). 
$$
The Christoffel symbols all vanish at $p$ except for 
\begin{equation}\label{christoffel}
\Ga_{\al\be\al'}=L_{\al\be\al'},\qquad\qquad
\Ga_{\al\al'\be}=\Ga_{\al'\al\be}=-L_{\al\be\al'}.
\end{equation}  
So specializing the indices and evaluating at $p$ gives
$$
-2R_{\al\al'\be\be'}=g_{\al\be,\al'\be'}+g_{\al'\be',\al\be}-g_{\al\be',\be\al'}-g_{\be\al',\al\be'} 
-2g^{\ga\de}L_{\al\ga\be'}L_{\be\de\al'}.
$$
Symmetrizing in $\al\be$ gives
$$
2R_{\al'(\al\be)\be'}=g_{\al\be,\al'\be'}+
\Sym_{\al\be}\big[(g_{\al'\be',\be}-g_{\be\be',\al'}-g_{\be\al',\be'}),{}_\al\big] 
-L_{\al\ga\be'}L_\be{}^\ga{}_{\al'}-L_{\be\ga\be'}L_\al{}^\ga{}_{\al'}.  
$$
But 
\begin{equation}\label{christderiv}
(g_{\al'\be',\be}-g_{\be\be',\al'}-g_{\be\al',\be'}){}_{,\al}=0 \qquad
  \text{on } \Si 
\end{equation}
since $g_{\al'\be',\be}-g_{\be\be',\al'}-g_{\be\al',\be'}=0$ on
$\Si$ by \eqref{gderivs}, and $\pa_\al$ acts tangentially.  Thus 
\eqref{R} holds.   

For \eqref{hess}, write 
$\na_\be v^{\al'}=v^{\al'}{}_{,\be}+\Ga_{\be\be'}^{\al'} v^{\be'}$, apply
$\na_{\al}$, and expand the right-hand side again in terms of partial
derivatives and Christoffel symbols.  Using the fact that all Christoffel
symbols vanish at $p$ except for \eqref{christoffel}, one obtains at $p$:
$$
\na_\al\na_\be
v^{\al'}=v^{\al'}{}_{,\be\al}+\Ga_{\be\be'}^{\al'}{}_{,\al}v^{\be'}
=v^{\al'}{}_{,\be\al}+\frac12 
g^{\al'\ga'}\big(g_{\ga'\be',\be}+g_{\be\ga',\be'}-g_{\be\be',\ga'}\big){}_{,\al}v^{\be'}.  
$$
But \eqref{christderiv} gives 
$g_{\ga'\be',\be\al}=\big(g_{\be\be',\ga'}+g_{\be\ga',\be'}\big){}_{,\al}$ 
on $\Si$, so substituting yields \eqref{hess}.  
\end{proof}

\noindent
{\it Proof of Proposition~\ref{a's}.} 
Since the definition \eqref{invariantexpand} of the $a^{(2j)}$ 
and the formulae in Proposition~\ref{a's} are coordinate-invariant, it   
suffices to prove them at $p$ in the coordinates $(x,u)$ 
constructed above.  Choose the $x$ coordinates on $\Si$ so that $x(p)=0$.
Then $p$ is represented by $(0,0,0)$ in the coordinates $(x,u,r)$ on $X$.  

The induced metric components are given by \eqref{hbar}, and according to  
\eqref{areaexpand}, we need to calculate the Taylor expansion of 
$\sqrt{\det \hb(x,r)}$.  Now $g_{ij}$ is given (\cite{FG}) by
\begin{equation}\label{poincare}
g_{ij}(x,u,r)=g_{ij}(x,u,0) -P_{ij}(x,u,0)r^2 + \wt{B}_{ij}(x,u,0)r^4+\ldots,
\end{equation}
where
\begin{equation}\label{bach}
\wt{B}_{ij}=\frac14\left[\frac{B_{ij}}{4-n}+P_i{}^kP_{kj}\right].
\end{equation}
By \eqref{U2}, we have 
\begin{equation}\label{u}
u^{\al'}(x,r)=\tfrac{1}{2k}H^{\al'}(x)r^2+u_{(4)}^{\al'}(x)r^4+\ldots.
\end{equation}
All $g_{ij}$ in \eqref{hbar} are evaluated at $(x,u(x,r),r)$.  Since 
$g_{\al\al'}(x,0,0)=0$, it follows from \eqref{poincare} and \eqref{u} 
that $g_{\al\al'}(x,u(x,r),r)=O(r^2)$.  Then \eqref{hbar} shows  
that $\hb_{\al 0}=O(r^3)$.  Thus
\begin{equation}\label{dethb}
\det\hb = (\det \hb_{\al\be})\cdot \hb_{00} +O(r^6).
\end{equation}

Staring at \eqref{hbar} and recalling that $u=O(r^2)$, it is clear 
that in order to evaluate $\hb_{\al\be}$ and $\hb_{00}$ through order
$r^4$, we need to know $g_{\al\be}$ through order $r^4$ and $g_{\al\al'}$
and $g_{\al'\be'}$ through order $r^2$.  Taking $x=0$, corresponding to the 
point $p$, we have 
\begin{equation}\label{gprime}
\begin{split}
g_{\al'\be'}(0,u(0,r),r)&=g_{\al'\be'}(0,u(0,r),0)-P_{\al'\be'}r^2+o(r^2)\\
&=g_{\al'\be'}(0,0,0)+g_{\al'\be',\ga'}(0,0,0)u^{\ga'}(0,r)
-P_{\al'\be'}r^2+o(r^2)\\
&=g_{\al'\be'}(0,0,0)
-P_{\al'\be'}r^2+o(r^2),
\end{split}
\end{equation}
where we used $g_{\al'\be',\ga'}(0,0,0)=0$ from Lemma~\ref{derivsvanish}.  
Likewise
$$
g_{\al\al'}(0,u(0,r),r)=-P_{\al\al'}r^2+o(r^2).
$$
Now \eqref{poincare} gives 
\begin{equation}\label{gexpand}
g_{\al\be}(0,u(0,r),r)=g_{\al\be}(0,u(0,r),0)
-P_{\al\be}(0,u(0,r),0)r^2+\wt{B}_{\al\be}r^4+o(r^4).
\end{equation}
Expanding the first term and substituting from Lemma~\ref{derivsvanish} and
\eqref{u} gives 
\[
\begin{split}
g_{\al\be}(0,u(0,r),0)&=
g_{\al\be}(0,0,0)+g_{\al\be,\al'}(0,0,0)u^{\al'}(0,r)
+\tfrac12 g_{\al\be,\al'\be'}(0,0,0)u^{\al'}u^{\be'}+o(r^4)\\
&=g_{\al\be}(0,0,0)-2L_{\al\be\al'}\left(\tfrac{1}{2k}H^{\al'}r^2
+u_{(4)}^{\al'}r^4\right)\\
&\qquad +\tfrac{1}{4k^2}\left(R_{\al'\al\be\be'}
+L_{\al\ga\be'}L_\be{}^\ga{}_{\al'}\right)H^{\al'}H^{\be'}r^4+o(r^4)\\  
&=g_{\al\be}(0,0,0)-\tfrac{1}{k}L_{\al\be\al'}H^{\al'}r^2\\
&\qquad +\left(-2L_{\al\be\al'}u_{(4)}^{\al'}+\tfrac{1}{4k^2}   
\left(R_{\al'\al\be\be'}
+L_{\al\ga\be'}L_\be{}^\ga{}_{\al'}\right)H^{\al'}H^{\be'}\right)r^4+o(r^4). 
\end{split}
\]
For use in the second term in \eqref{gexpand}, we have 
$$
P_{\al\be}(0,u(0,r),0)=P_{\al\be}(0,0,0) +P_{\al\be,\al'}u^{\al'}+o(r^2)
=P_{\al\be}(0,0,0) +\tfrac{1}{2k}P_{\al\be,\al'}H^{\al'}r^2+o(r^2).
$$
Substituting these into \eqref{gexpand} gives
\[
\begin{split}
g_{\al\be}&(0,u(0,r),r)=g_{\al\be}(0,0,0)
+\left(-\tfrac{1}{k}L_{\al\be\al'}H^{\al'}-P_{\al\be}(0,0,0)\right)r^2\\  
&+\left(-2L_{\al\be\al'}u_{(4)}^{\al'}+\tfrac{1}{4k^2} 
\left(R_{\al'\al\be\be'}
+L_{\al\ga\be'}L_\be{}^\ga{}_{\al'}\right)H^{\al'}H^{\be'}-\tfrac{1}{2k}P_{\al\be,\al'}H^{\al'} 
+\wt{B}_{\al\be}\right)r^4+o(r^4).
\end{split}
\]

Now substitute all these into \eqref{hbar}.  Henceforth, all $g_{ij}$ and 
$P_{ij}$ are understood to be evaluated at $p$.  One obtains
\begin{equation}\label{hb}
\begin{split}
\hb_{\al\be}=&g_{\al\be}+D_{\al\be}r^2+Q_{\al\be}r^4+o(r^4)\\
\hb_{00}=&1+Er^2 +F r^4+o(r^4)
\end{split}
\end{equation}
with
\begin{equation}\label{hbcoeffs}
\begin{split}
D_{\al\be}=&-\tfrac{1}{k}L_{\al\be}^{\al'}H_{\al'}-P_{\al\be}\\
Q_{\al\be}=&-2L_{\al\be\al'}u_{(4)}^{\al'}
+\tfrac{1}{4k^2} R_{\al'\al\be\be'}H^{\al'}H^{\be'}
+\tfrac{1}{4k^2} L_{\al\ga}^{\al'}L_{\be}{}^\ga{}_{\be'}H_{\al'}H^{\be'}\\
&-\tfrac{1}{2k}P_{\al\be,\al'}H^{\al'}
+\wt{B}_{\al\be}
-\tfrac{1}{k}P_{\al'(\al} H^{\al'}{}_{,\be)}
+\tfrac{1}{4k^2}g_{\al'\be'}H^{\al'}{}_{,\al}H^{\be'}{}_{,\be}\\
E=&\quad\tfrac{1}{k^2}|H|^2\\
F=&-\tfrac{1}{k^2}P_{\al'\be'}H^{\al'}H^{\be'}
+\tfrac{8}{k}H_{\al'}u_{(4)}^{\al'}.
\end{split}
\end{equation}
However, again by Lemma~\ref{derivsvanish} and using \eqref{christoffel},
at $p$ we have $H^{\al'}{}_{,\al}=\nabla_\al H^{\al'}$ and 
$P_{\al\be,\al'}={}^M\na_{\al'}P_{\al\be} -L_{\al\ga\al'}P^\ga{}_{\be}
-L_{\be\ga\al'}P^\ga{}_{\al}$.  So 
$P_{\al\be,\al'}H^{\al'}={}^M\na_{\al'}P_{\al\be}H^{\al'}
-2L_{\ga(\al}^{\al'}P_{\be)}^{\vphantom{\al'}}{}^\ga H_{\al'}$.  Thus the
formula for $Q_{\al\be}$ above becomes  
\begin{equation}\label{Q}
\begin{split}
Q_{\al\be}=&-2L_{\al\be\al'}u_{(4)}^{\al'}
+\tfrac{1}{4k^2} R_{\al'\al\be\be'}H^{\al'}H^{\be'}
+\tfrac{1}{4k^2} L_{\al\ga}^{\al'}L_{\be}{}^\ga{}_{\be'}H_{\al'}H^{\be'}
-\tfrac{1}{2k}{}^M\na_{\al'}P_{\al\be}H^{\al'}\\
&+\tfrac{1}{k}L_{\ga(\al}^{\al'}P_{\be)}^{\vphantom{\al'}}{}^\ga H_{\al'}
+\wt{B}_{\al\be}
-\tfrac{1}{k}P_{\al'(\al} \na_{\be)} H^{\al'}
+\tfrac{1}{4k^2}g_{\al'\be'}\na_\al H^{\al'}\na_{\be}H^{\be'}.
\end{split}
\end{equation}

We need to calculate $\det{\hb_{\al\be}}$ to use in \eqref{dethb}.  Taylor
expanding the determinant function shows that for  
$\hb_{\al\be}$ of the form \eqref{hb}, we have 
$$
\det{\hb_{\al\be}}
=\det{g_{\al\be}}\left[1+D_\al{}^\al r^2
+\left(Q_\al{}^\al
-\tfrac12 D_{\al\be}D^{\al\be} 
+\tfrac12 (D_\al{}^\al)^2\right)r^4+o(r^4)\right].  
$$
Multiplying by $\hb_{00}$ and recalling \eqref{dethb}, we get 
\begin{equation}\label{det}
\begin{split}
\frac{\det{\hb}}{\det{g_{\al\be}}}
&=\frac{\det{\hb_{\al\be}}}{\det{g_{\al\be}}}\cdot \hb_{00}+O(r^6)\\
&=1+\left(D_\al{}^\al+E\right)r^2 +\left(Q_\al{}^\al-\tfrac12
D_{\al\be}D^{\al\be} +\tfrac12 (D_\al{}^\al)^2+F+ED_\al{}^\al\right)r^4
+o(r^4). 
\end{split}
\end{equation}
Finally, using $\sqrt{1+x}=1+\frac12 x -\frac18 x^2 +o(x^2)$ gives 
\[
\begin{split}
\sqrt{\frac{\det{\hb}}{\det{g_{\al\be}}}}
= 1 &+\tfrac12 \left(D_\al{}^\al+E\right)r^2 \\
&+\tfrac12 \left(Q_\al{}^\al-\tfrac12 D_{\al\be}D^{\al\be} 
+\tfrac14 (D_\al{}^\al)^2+F+\tfrac12 ED_\al{}^\al-\tfrac14 E^2\right)r^4
+o(r^4).
\end{split}
\]
Recalling \eqref{areaexpand} and that $\hb_{\al\be}=g_{\al\be}$ when 
$r=0$, we conclude  
\[
\begin{split}
a^{(2)}&= \tfrac12 \left(D_\al{}^\al+E\right)\\
a^{(4)}&= \tfrac12 \left(Q_\al{}^\al-\tfrac12 D_{\al\be}D^{\al\be} 
+\tfrac14 (D_\al{}^\al)^2+F+\tfrac12 ED_\al{}^\al-\tfrac14 E^2\right).  
\end{split}
\]
The formula for $a^{(2)}$ in Proposition~\ref{a's} follows upon
substituting \eqref{hbcoeffs} for $D_{\al\be}$ and $E$.  To obtain the
formula for $a^{(4)}$, one 
substitutes \eqref{hbcoeffs} for $D_{\al\be}$, $E$, and $F$, \eqref{Q} for 
$Q_{\al\be}$, \eqref{bach} for $\wt{B}_{\al\be}$ in \eqref{Q}, writes 
$R_{\al'\al\be\be'}$ in \eqref{Q} in terms of the Weyl and Schouten tensors
via \eqref{curv}, and collects terms.    
\stopthm

\begin{corollary}\label{Eformcor}
If $k=2$, then 
$$
\cE=-\frac18 \int_\Si \big(|H|^2 +4P^\al{}_\al\big)\,\,da_\Si
$$
If $k=4$, then 
\[
\begin{split}
\cE=\frac{1}{128}\int_\Si \Bigg(& |\nabla H|^2
-L_{\al\be}^{\al'}L^{\al\be}_{\be'}H_{\al'}H^{\be'}
+\frac{7}{16} |H|^4\\
&-W^\al{}_{\al'\al\be'}H^{\al'}H^{\be'}-8P^{\al}{}_{\al'}\nabla_\al
H^{\al'}
-8C^\al{}_{\al\al'}H^{\al'}-8P^{\al\be}L^{\al'}_{\al\be}H_{\al'} 
+5P^\al{}_{\al}|H|^2 \\
&-16P^{\al\be}P_{\al\be}+16 P^{\al\al'}P_{\al\al'}
+16(P^\al{}_\al)^2 -\frac{16}{n-4}B^\al{}_\al\Bigg)\,\,da_\Si
\end{split}
\]
\end{corollary}
\begin{proof}
Recall that $\cE=\int_\Si a^{(k)} da_\Si$ for $\Si$ of dimension $k$.  
The result for $k=2$ follows immediately upon setting $k=2$ in
Proposition~\ref{a's} and integrating.  For $k=4$, integrating the formula
in Proposition~\ref{a's} and comparing with that in
Corollary~\ref{Eformcor} shows that the result
reduces to the following identity:
\begin{equation}\label{int}
\begin{split}
\int_\Si \Big(
g^{\al\be}{}^M\na_{\al'}P_{\al\be}H^{\al'}+P^\al{}_{\al'}\na_{\al}H^{\al'} 
&+P_{\al'\be'}H^{\al'}H^{\be'}\Big)da_\Si \\
&=
\int_\Si
\left(C^\al{}_{\al\al'}H^{\al'}+P^{\al\be}L_{\al\be}^{\al'}H_{\al'}\right) da_\Si.
\end{split}
\end{equation}
We intend to integrate the $\na_\al$ by parts in the second term on the
left-hand side to obtain $-\na_\al P^\al{}_{\al'} H^{\al'}$.   
This $\na_\al$ denotes the normal bundle connection, so we are obliged here
to interpret $P^\al{}_{\al'}$ as a section of $T\Si\otimes N^*\Si$ on 
$\Si$ and $\na_\al$ as the induced connection on this bundle.
Recalling \eqref{christoffel}, we have 
$$
\na_\al P_{\be\al'}={}^M\na_\al P_{\be\al'}
+\Ga_{\al\be}^{\be'}P_{\be'\al'} +\Ga_{\al\al'}^\ga P_{\be\ga}
={}^M\na_\al P_{\be\al'}
+L_{\al\be}^{\be'}P_{\be'\al'} -L_{\al\al'}^\ga P_{\be\ga}.
$$
Thus
$$
-\na_\al P^\al{}_{\al'} H^{\al'} = -g^{\al\be}{}^M\na_\al P_{\be\al'} H^{\al'}
-P_{\al'\be'}H^{\al'}H^{\be'}+L_{\al\be}^{\al'} P^{\al\be}H_{\al'}.
$$
So integrating by parts as described and substituting for    
$-\na_\al P^\al{}_{\al'} H^{\al'}$, one concludes that the
left-hand side of \eqref{int} equals
\[
\begin{split}
\int_\Si \Big(
g^{\al\be}{}^M\na_{\al'}P_{\al\be}H^{\al'}
-g^{\al\be}{}^M\na_\al P_{\be\al'} H^{\al'}
&+L_{\al\be}^{\al'} P^{\al\be}H_{\al'}\Big)
da_\Si \\
&=\int_\Si 
\left(C^\al{}_{\al\al'}H^{\al'}+P^{\al\be}L_{\al\be}^{\al'}H_{\al'}\right)
da_\Si
\end{split}
\]
as desired.
\end{proof}
\begin{remark}
The coefficients in Proposition~\ref{a's} and Corollary~\ref{Eformcor}
become simpler when written in terms of the alternate convention for the 
mean curvature:  $\Hb = \frac{1}{k}H$.  In particular, one has 
\[
\begin{split}
a^{(4)} =\frac{1}{8} & \Bigg( |\nabla \Hb|^2 -
L^{\al'}_{\al\be}L^{\al\be}_{\be'}\Hb_{\al'}\Hb^{\be'}
+(k^2-2k-1)|\Hb|^4\\
& -W^\al{}_{\al'\al\be'}\Hb^{\al'}\Hb^{\be'}  
-2 g^{\al\be}{}^M\nabla_{\al'}P_{\al\be} \Hb^{\al'}   
-4 P^\al{}_{\al'}\nabla_\al \Hb^{\al'}\\ 
&+(2k-3)P^\al{}_\al|\Hb|^2 -(k+4)P_{\al'\be'}\Hb^{\al'}\Hb^{\be'} \\
& -P^{\al\be}P_{\al\be}+ P^{\al\al'}P_{\al\al'}+(P^\al{}_\al)^2 
-\frac{1}{n-4}B^\al{}_\al\Bigg)\\
&+(4-k)\Hb_{\al'}U^{\al'}_{(4)} 
\end{split}
\]
and for $k=4$
\[
\begin{split}
\cE=\frac{1}{8}\int_\Si \Big(& |\nabla \Hb|^2
-L_{\al\be}^{\al'}L^{\al\be}_{\be'}\Hb_{\al'}\Hb^{\be'}
+7 |\Hb|^4\\
&-W^\al{}_{\al'\al\be'}\Hb^{\al'}\Hb^{\be'}-2P^{\al}{}_{\al'}\nabla_\al
\Hb^{\al'}
-2C^\al{}_{\al\al'}\Hb^{\al'}-2P^{\al\be}L^{\al'}_{\al\be}\Hb_{\al'} 
+5P^\al{}_{\al}|\Hb|^2 \\
&-P^{\al\be}P_{\al\be}+ P^{\al\al'}P_{\al\al'}
+(P^\al{}_\al)^2 -\frac{1}{n-4}B^\al{}_\al\Big)\,\,da_\Si.
\end{split}
\]
\end{remark}

Recall that equation~\eqref{U2} identifies $U_{(2)}$.  The next proposition
identifies $U_{(4)}$.  

\begin{proposition}\label{U4prop}
If $k>2$, then
\[
\begin{split}
U_{(4)}^{\al'}=\frac{1}{8k(k-2)}&
\Bigg((\Delta H)^{\al'}
+L^{\al'}_{\al\be}L^{\al\be}_{\be'}H^{\be'}-\frac{2}{k^2}|H|^2H^{\al'} \\
&+W^{\al\al'}{}_{\al\be'}{}H^{\be'}
-P^\al{}_\al H^{\al'}+(k-4)P^{\al'}{}_{\be'} H^{\be'} 
+2kP^{\al\be}L_{\al\be}^{\al'}\\
&+k g^{\al'\be'}g^{\al\be}({}^M\nabla_{\be'}P_{\al\be}-2{}^M\nabla_\al
P_{\be\be'})\Bigg)  
\end{split}
\]

If $k=2$, then
\[
\begin{split}
\cH^{\al'}=\frac14 & \Bigg((\Delta H)^{\al'}
+L^{\al'}_{\al\be}L^{\al\be}_{\be'}H^{\be'}-\frac12|H|^2H^{\al'} \\
&+W^{\al\al'}{}_{\al\be'}{}H^{\be'}
-P^\al{}_\al H^{\al'}-2P^{\al'}{}_{\be'} H^{\be'} 
+4P^{\al\be}L_{\al\be}^{\al'}\\
&+2 g^{\al'\be'}g^{\al\be}({}^M\nabla_{\be'}P_{\al\be}-2{}^M\nabla_\al
P_{\be\be'})\Bigg)  
\end{split}
\]
\end{proposition}
\begin{proof}
Just as in the proof of Proposition~\ref{a's}, since these formulas are  
coordinate-invariant, it suffices to prove them at $p$ in our special
adapted coordinates $(x,u)$.  Now $U_{(4)}$ is determined by applying
$\pa_r^3|_{r=0}$ to the equation $\cM(u)=0$, with $\cM(u)$ given by 
\eqref{Mform}.  Recall that in \eqref{Mform}, we have $\hb^{0\al}=O(r^3)$, 
$g_{\al\ga'}=O(r^2)$ and $u^{\al'}=O(r^2)$.  Therefore  
\[
\begin{aligned}
\cM (u)_{\ga'} = &
\left( r \partial_{r} - (k+1) + \frac{1}{2} r\cL_{,r} \right)
\left( \hb^{00} g_{\al'\ga'}u^{\al'}{}_{,r} 
\right)\\
& +r\left( \partial_{\be} + \frac{1}{2} \cL_{,\be} \right)
\left [ 
\hb^{\al\be} \left ( g_{\al\ga'} + g_{\al'\ga'}u^{\al'}{}_{,\al} 
\right ) \right ] \\
& -\frac{1}{2} r \hb^{\al\be}\left( g_{\al\be,\ga'} +
2g_{\al\al',\ga'}u^{\al'}{}_{,\be} \right) 
 -\frac{1}{2} r \hb^{00}
g_{\al'\be',\ga'}u^{\al'}{}_{,r}u^{\be'}{}_{,r}  +O(r^5).
\end{aligned}
\]
Apply $\pa_r^3|_{r=0}$ to the equation $\cM(u)=0$.  Keeping in mind the
orders of vanishing and the parity of the various terms and the fact that
$\hb^{00}=1$ on $\Si$, one obtains at $r=0$: 
\begin{equation}\label{alongtheway}
\begin{aligned}
0=& (2-k)\pa_r^3 \left( \hb^{00} g_{\al'\ga'}u^{\al'}{}_{,r} \right)
+3\cL_{,rr}g_{\al'\ga'}u^{\al'}{}_{,rr}\\ 
& +3\left( \partial_{\be} + \frac{1}{2} \cL_{,\be} \right)
\left [ 
\hb^{\al\be} \left ( \pa_r^2 g_{\al\ga'} + g_{\al'\ga'}u^{\al'}{}_{,\al rr}  
\right ) \right ] \\
& -\frac{3}{2}  \hb^{\al\be}\left( \pa_r^2 g_{\al\be,\ga'}
 + 2g_{\al\al',\ga'}u^{\al'}{}_{,\be rr} \right) 
-\frac32 \left(\pa_r^2 \hb^{\al\be}\right)g_{\al\be,\ga'} 
 -3g_{\al'\be',\ga'}u^{\al'}{}_{,rr}u^{\be'}{}_{,rr}.
\end{aligned}
\end{equation}
Expanding the derivatives gives
$$
\pa_r^3 \left( \hb^{00} g_{\al'\ga'}u^{\al'}{}_{,r}\right)
=g_{\al'\ga'}\pa_r^4 u^{\al'}+
3\left(g_{\al'\ga'}\pa_r^2 \hb^{00}+\pa_r^2g_{\al'\ga'} \right)
u^{\al'}{}_{,rr} 
$$
at $r=0$.
Since $\hb_{\al\be}=g_{\al\be}$ on $\Sigma$, it follows from
Lemma~\ref{derivsvanish} that 
$\cL_{,\be}=g^{\al\ga}g_{\al\ga,\be}=0$ at $p$.  Lemma~\ref{derivsvanish}
also implies that $\pa_\be \hb^{\al\be}$, 
$g_{\al\al',\ga'}$ and $g_{\al'\be',\ga'}$ vanish at $p$.  
So evaluating \eqref{alongtheway} at $p$ and solving for
$\pa_r^4 u^{\al'}$ yield
\begin{align*}
(k-2)g_{\al'\ga'}\pa_r^4 u^{\al'}=&
3(2-k)\left(g_{\al'\ga'}\pa_r^2 \hb^{00}+\pa_r^2g_{\al'\ga'}
\right)u^{\al'}{}_{,rr} 
+3\cL_{,rr}g_{\al'\ga'}u^{\al'}{}_{,rr}\\ 
& +3
g^{\al\be} \partial_{\be} \left ( \pa_r^2 g_{\al\ga'} 
+ g_{\al'\ga'}u^{\al'}{}_{,\al rr}  \right ) \\
& -\frac{3}{2}  g^{\al\be} \pa_r^2 g_{\al\be,\ga'}
-\frac32 \left(\pa_r^2 \hb^{\al\be}\right)g_{\al\be,\ga'}.
\end{align*}
Expanding the derivative on the second line and reordering terms gives 
\begin{equation}\label{furtheralong}
\begin{aligned}
(k-2)g_{\al'\ga'}\pa_r^4 u^{\al'}=&
3g_{\al'\ga'}g^{\al\be}u^{\al'}{}_{,rr\al\be} 
+3g^{\al\be}  \pa_\be\pa_r^2 g_{\al\ga'}  \\
&+
3(2-k)g_{\al'\ga'}\left(\pa_r^2 \hb^{00}\right)u^{\al'}{}_{,rr}
+3(2-k)\left(\pa_r^2g_{\al'\ga'}\right)u^{\al'}{}_{,rr} \\ 
&+3\cL_{,rr}g_{\al'\ga'}u^{\al'}{}_{,rr}
 -\frac{3}{2}  g^{\al\be} \pa_r^2 g_{\al\be,\ga'}
-\frac32 \left(\pa_r^2 \hb^{\al\be}\right)g_{\al\be,\ga'}.
\end{aligned}
\end{equation}

We have already evaluated many of the ingredients on the right-hand side.
For instance, $g_{\al\be,\ga'}=-2L_{\al\be\ga'}$ by \eqref{whichone} and
$\pa_r^2u^{\al'}=\frac{1}{k}H^{\al'}$ by \eqref{U2}.  Also
$\pa_r^2\hb^{00}=-2E$ and $\pa_r^2\hb^{\al\be}=-2D^{\al\be}$
by \eqref{hb}.  Similarly, taking the log and differentiating in
\eqref{det} gives $\cL_{,rr}=2(D_\al{}^\al +E)$.  We have 
$\pa_r^2 g_{\al'\ga'}=-2P_{\al'\ga'}$ by \eqref{gprime}.  The terms that
still need to be evaluated are $g^{\al\be}u^{\al'}{}_{,rr\al\be}$,
$\pa_\be\pa_r^2 g_{\al\ga'}$ and $\pa_r^2 g_{\al\be,\ga'}$.  Using
\eqref{hess}, we have at $p$:
$$
g^{\al\be}u^{\al'}{}_{,rr\al\be}=\frac{1}{k}g^{\al\be}H^{\al'}{}_{,\al\be}
=\frac{1}{k}\Delta H^{\al'}
-\frac{1}{k}g^{\al'\ga'}g^{\al\be}g_{\al\ga',\be\be'}H^{\be'}.    
$$
For $\pa_\be\pa_r^2 g_{\al\ga'}$ and $\pa_r^2 g_{\al\be,\ga'}$, recall that
in \eqref{Mform}, all components of 
$g$ and its derivatives which appear are evaluated at $(x,u(x,r,r)$.   We
can evaluate $\pa_r^2 g_{\al\be,\ga'}$ by the same procedure we used in the
proof of Proposition~\ref{a's}.  First differentiate \eqref{poincare} to
obtain 
$$
g_{\al\be,\ga'}(x,u,r)=g_{\al\be,\ga'}(x,u,0)-P_{\al\be,\ga'}r^2 +o(r^2).
$$
Evaluate at $u=u(x,r)$:  
\[
\begin{split}
g_{\al\be,\ga'}(x,u(x,r),r)=&g_{\al\be,\ga'}(x,u(x,r),0)-P_{\al\be,\ga'}r^2
+o(r^2)\\
=&g_{\al\be,\ga'}(x,0,0)+g_{\al\be,\ga'\be'}(x,0,0)u^{\be'}(x,r)-P_{\al\be,\ga'}r^2 +o(r^2) \\
=&g_{\al\be,\ga'}(x,0,0)+\frac{1}{2k}g_{\al\be,\ga'\be'}H^{\be'}r^2-P_{\al\be,\ga'}r^2+o(r^2).  
\end{split}
\]
Thus 
$$
\pa_r^2
g_{\al\be,\ga'}=\frac{1}{k}g_{\al\be,\ga'\be'}H^{\be'}-2P_{\al\be,\ga'}.
$$
For $\pa_\be\pa_r^2 g_{\al\ga'}$, begin as above by differentiating
\eqref{poincare}: 
$$
g_{\al\ga',\be}(x,u,r)=g_{\al\ga',\be}(x,u,0)-P_{\al\ga',\be}r^2 +o(r^2).
$$
So the chain rule gives
\[
\begin{split}
\pa_\be\big( g_{\al\ga'}&(x,u(x,r),r)\big)=g_{\al\ga',\be}(x,u(x,r),r)+ 
g_{\al\ga',\be'}(x,u(x,r),r)u^{\be'}{}_{,\be}(x,r)\\
=&g_{\al\ga',\be}(x,u(x,r),0)-P_{\al\ga',\be}r^2
+\frac{1}{2k}g_{\al\ga',\be'}(x,0,0)H^{\be'}{}_{,\be}r^2 +o(r^2)\\
=&g_{\al\ga',\be}(x,0,0)+g_{\al\ga',\be\be'}u^{\be'}(x,r)
-P_{\al\ga',\be}r^2
+\frac{1}{2k}g_{\al\ga',\be'}(x,0,0)H^{\be'}{}_{,\be}r^2 +o(r^2)\\
=&g_{\al\ga',\be}(x,0,0)+\frac{1}{2k}g_{\al\ga',\be\be'}H^{\be'}r^2
-P_{\al\ga',\be}r^2
+\frac{1}{2k}g_{\al\ga',\be'}(x,0,0)H^{\be'}{}_{,\be}r^2 +o(r^2).
\end{split}
\]
At $x=0$, corresponding to the point $p$, we have $g_{\al\ga',\be'}=0$.
Thus at $p$ we obtain
$$
\pa_\be\pa_r^2
g_{\al\ga'}=\frac{1}{k}g_{\al\ga',\be\be'}H^{\be'}-2P_{\al\ga',\be}.
$$

Now substitute all of these into \eqref{furtheralong} and raise the free
index.  The two terms involving $g_{\al\ga',\be\be'}$ cancel, and one
obtains
\begin{equation}\label{lastone}
\begin{aligned}
(k-2)\pa_r^4 u^{\al'}=&
\frac{3}{k}\Delta
H^{\al'}
+3g^{\al'\ga'}g^{\al\be}\big(P_{\al\be,\ga'} -2P_{\al\ga',\be} \big)
-\frac{6(2-k)}{k}E H^{\al'}
-\frac{6(2-k)}{k}P^{\al'}{}_{\be'}H^{\be'} \\ 
&+\frac{6}{k}\big(D^\al{}_{\al}+E\big)H^{\al'}
 -\frac{3}{2k} g^{\al'\ga'}g^{\al\be}g_{\al\be,\ga'\be'}H^{\be'}
-6 D^{\al\be}L_{\al\be}^{\al'}. 
\end{aligned}
\end{equation}
Expanding the covariant derivatives in terms of partial derivatives and
Christoffel symbols shows that at $p$:
\begin{equation}\label{sub}
g^{\al\be}\big(P_{\al\be,\ga'} -2P_{\al\ga',\be}\big)
=g^{\al\be}\big({}^M\nabla_{\ga'}P_{\al\be}-2{}^M\nabla_\be
P_{\al\ga'}\big) -2P_{\ga'\be'}H^{\be'}.
\end{equation}
The formula for $U_{(4)}$ in Proposition~\ref{U4prop} is obtained as
follows.  In \eqref{lastone}, substitute \eqref{sub} for
$g^{\al\be}\big(P_{\al\be,\ga'} -2P_{\al\ga',\be}\big)$, substitute
\eqref{hbcoeffs} for $D_{\al\be}$ and $E$, substitute \eqref{R} for
$g_{\al\be,\ga'\be'}$, write the resulting $R_{\ga'\al\be\be'}$ in terms of
the Weyl and Schouten tensors via \eqref{curv}, collect terms, and finally
note that $U_{(4)}=\frac{1}{24}\pa_r^4U|_{r=0}$.    

Equation \eqref{lastone} confirms that $\pa_r^3 \cM(u)|_{r=0}$ is
independent of $\pa_r^4 u^{\al'}|_{r=0}$ when $k=2$.  In this case, the above
calculation shows that $g^{\al'\ga'}\pa_r^3\cM(u)_{\ga'}|_{r=0}$ equals the 
right-hand side of \eqref{lastone}.  However, by \eqref{cH} we have
$$
g^{\al'\ga'}\pa_r^3\cM(u)_{\ga'}|_{r=0}=\pa_r^3(r^3\cH^{\al'})|_{r=0} =
6\cH^{\al'}.  
$$
This gives the formula for $\cH$ in Proposition~\ref{U4prop}.  
\end{proof} 

We have also calculated the variation for  
$k=4$ when the background $M$ is $\R^n$ with the Euclidean metric.  In 
\cite{Gu1}, \cite{Gu2}, 
Guven developed a very nice formalism for identifying the variation  
of functionals of hypersurfaces in Euclidean space, using Lagrange      
multipliers to encode the data of the embedding.  As he suggested, it is 
straightforward to extend his formalism to submanifolds of higher
codimension.  We refer to his papers for the details of the 
method and formulate the following consequences.  

Let $S(g_{\al\be},g_{\al'\be'},L_{\al\be\al'})$ be a smooth function of 
$g_{\al\be}\in S^2_+\R^k{}^*$, $g_{\al'\be'}\in S^2_+\R^{n-k}{}^*$, and 
$L_{\al\be\al'}\in S^2\R^k{}^*\otimes \R^{n-k}{}^*$.  Here $S^2_+\R^l{}^*$ 
denotes the cone of positive definite symmetric bilinear forms on $\R^l$.  
It is assumed that $S$ is invariant under the natural action of
$GL(k,\R)\times GL(n-k,\R)$ on
$S^2_+\R^k{}^*\oplus S^2_+\R^{n-k}{}^*\oplus (S^2\R^k{}^*\otimes
\R^{n-k}{}^*)$.  Let $\Sigma$ be a compact $k$-manifold and let 
\begin{equation}\label{cF}
\cF = \int_\Sigma S(g_{\al\be},g_{\al'\be'},L_{\al\be\al'})\,da_\Si
\end{equation}
be a functional on the space of immersions of $\Si$ into $\R^n$, where at
each point 
$g_{\al\be}$ is taken to be the induced metric, $g_{\al'\be'}$ the metric
on the normal 
space, and $L_{\al\be\al'}$ the second fundamental form with all indices
lowered, all written in some choice of local frames for $T\Si$ and for 
$N\Si$.     
By the assumed invariance of $S$, the integrand is independent of the
choice of frames and is a globally defined density on $\Si$ depending
on the immersion $f$. 

As in Theorem~\ref{variation}, let $F:\Si\times [0,\de)\rightarrow \R^n$ be
a variation of $\Si$ with infinitesimal variation $\dot{F}\in \Ga(f^*TM)$.
Guven's method shows that 
$$
\dot{\cF} = \int_\Sigma \langle \dot{F},\cK\rangle_g\,da_\Si,
$$
where $\cK$ is the section of $N\Si$ given by
\begin{equation}\label{guven}
\cK^{\al'}=\na_\al\na_\be\Big(\frac{\pa S}{\pa L_{\al\be\al'}}\Big)
+L_{\al\ga}^{\al'}L^\ga_{\be\be'}\frac{\pa S}{\pa L_{\al\be\be'}}
-SH^{\al'}.
\end{equation}
For instance, for $S=1$, $\cF$ is the area functional and $\cK=-H$ is the 
negative of the mean curvature.  For another example, the integrand for the 
Willmore energy is 
$$
S= |H|^2= g^{\al'\be'}g^{\al\be}g^{\ga\de}L_{\al\be\al'}L_{\ga\de\be'}.
$$
So $\frac{\pa S}{\pa L_{\al\be\al'}}=2g^{\al\be}H^{\al'}$, 
and \eqref{guven} gives the usual formula for its variation:
$$
\cK^{\al'}=2\Delta H^{\al'}+2L_{\al\be}^{\al'}L^{\al\be}_{\be'}H^{\be'}
- |H|^2H^{\al'}.
$$

We apply Guven's formalism to the energy functional $\cE$ for $k=4$.  
If $M=\R^n$, then all background curvature vanishes, so the formula in 
Corollary~\ref{Eformcor} for the $k=4$ energy becomes \eqref{Eeuclidean}. 
Consider separately the variation of each of the
three terms in the integrand.  For $S=|H|^4$, we have 
$\frac{\pa S}{\pa L_{\al\be\al'}}=4g^{\al\be}|H|^2 H^{\al'}$, so
$$
\cK^{\al'}=4\Delta\big(|H|^2
H^{\al'}\big)+4L_{\al\be}^{\al'}L^{\al\be}_{\be'}|H|^2H^{\be'}
-|H|^4H^{\al'}.   
$$
For $S=L_{\al\be}^{\al'}L^{\al\be}_{\be'}H_{\al'}H^{\be'}$, we have
$\frac{\pa S}{\pa L_{\al\be\al'}}
=2g^{\al\be}L_{\ga\de}^{\al'}L^{\ga\de}_{\be'}H^{\be'}
+2L^{\al\be\be'}H_{\be'}H^{\al'}$,  
so
\[
\begin{split}
\cK^{\al'}&=2\Delta\Big(L_{\al\be}^{\al'}L^{\al\be}_{\be'}H^{\be'}\Big)
+2\na_\al\na_\be\Big(L^{\al\be\be'}H_{\be'}H^{\al'}\Big)\\
&+2L_{\ga\de}^{\al'}L^{\ga\de}_{\be'}L_{\al\be\ga'}L^{\al\be\be'}H^{\ga'}
+2L_{\al\ga}^{\al'}L^\ga_{\be\be'}L^{\al\be\ga'}H^{\be'}H_{\ga'}
-L_{\al\be}^{\ga'}L^{\al\be}_{\be'}H_{\ga'}H^{\be'}H^{\al'}.
\end{split}
\]
In addition to functionals of the form \eqref{cF}, Guven's formalism 
applies to functionals involving covariant derivatives of
$L_{\al\be}^{\al'}$.  For $S=|\na H|^2$, it gives (see 
\cite{Gu2}): 
$$
\cK^{\al'}
=-2\Delta^2H^{\al'}+2L^{\al\be\al'}\na_\al H^{\be'}\na_\be H_{\be'}
-2L^{\al\be\al'}L_{\al\be\be'}\Delta H^{\be'}-|\na H|^2H^{\al'}.
$$
Recalling that Theorem~\ref{variation} identifies $\cH$ as the negative of
the variation of $\cE$ and combining the above ingredients 
yields the following.
\begin{proposition}\label{Hk=4}
When $k=4$ and $M=\R^n$ with the Euclidean metric, we have
\[
\begin{split}
128\, \cH^{\al'}&=
2\Delta^2H^{\al'}-2L^{\al\be\al'}\na_\al H^{\be'}\na_\be H_{\be'}
+2L^{\al\be\al'}L_{\al\be\be'}\Delta H^{\be'}+|\na H|^2H^{\al'}\\
&+2\Delta\Big(L_{\al\be}^{\al'}L^{\al\be}_{\be'}H^{\be'}\Big)
+2\na_\al\na_\be\Big(L^{\al\be\be'}H_{\be'}H^{\al'}\Big)\\
&+2L_{\ga\de}^{\al'}L^{\ga\de}_{\be'}L_{\al\be\ga'}L^{\al\be\be'}H^{\ga'}
+2L_{\al\ga}^{\al'}L^\ga_{\be\be'}L^{\al\be\ga'}H^{\be'}H_{\ga'}
-L_{\al\be}^{\ga'}L^{\al\be}_{\be'}H_{\ga'}H^{\be'}H^{\al'}\\
&-\frac74 \Delta\big(|H|^2H^{\al'}\big)
-\frac74 L_{\al\be}^{\al'}L^{\al\be}_{\be'}|H|^2H^{\be'}
+\frac{7}{16}|H|^4H^{\al'}.
\end{split}
\]
\end{proposition}

\noindent
We checked this formula by evaluating it on the anchor ring embeddings
$T^{2,2}_{R,r}$ and $T^{3,1}_{R,r}$ discussed in 
\S\ref{anchor}.   It vanishes exactly for the values of $R$ and $r$
corresponding to the 
$\cE$-critical anchor ring embeddings described there.     

We close this section by noting that modulo linear terms in background
curvature and quadratic terms in derivatives of $g$, we have 
\[
\begin{split}
(2l)!\,U_{(2l)}^{\al'}&=\frac{1\cdot 3\cdot 5\cdots
  (2l-1)}{k(k-2)\cdots(k+2-2l)}\,\Delta^{l-1}H^{\al'} ,\qquad 
1\leq l\leq k/2\\
k!\,\cH^{\al'} &= \frac{1\cdot 3\cdot 5\cdots
  (k-1)}{k(k-2)\cdots 2}\,\Delta^{k/2}H^{\al'},\qquad k\geq 2
\end{split}
\]
These follow by induction keeping track only of the linear terms in the  
calculations indicated above.  Modulo terms involving background  
curvature and terms of higher homogeneity degree in the second fundamental
form and its derivatives, the energy has the form 
\[
\cE=c_k\int_\Si \langle H,\Delta^{(k-2)/2}H\rangle\,da_\Si,\qquad k\geq 2,
\qquad c_k\neq 0.   
\]

\section{Dimension Four}\label{4d} 

In this section we take $k=4$ throughout (so $n\geq 5$), and $M$ will be 
either $\R^n$ 
with the Euclidean metric or $S^n$ with the round metric of sectional 
curvature 1.  $\Sigma$ is a compact 4-dimensional immersed submanifold of
$M$.    

If $M=\R^n$, then all background curvature vanishes, so the formula in 
Corollary~\ref{Eformcor} for the $k=4$ energy becomes \eqref{Eeuclidean}.  
If $M=S^n$, then the Weyl, Cotton and Bach tensors  
all vanish, and the Schouten tensor is given by $P = \frac12 g$.  In this
case the formula in Corollary~\ref{Eformcor} reduces to \eqref{Esphere}. 
Set $\cEb = 128\, \cE$.  

Proposition~\ref{minimalcritical} shows that if $\Si$ is a 4-dimensional
immersed submanifold of either $\R^n$ or $S^n$ and $\Si$ is minimal, then
$\Si$ is critical for $\cE$.  As noted in the introduction, this also
follows immediately from \eqref{Eeuclidean} and \eqref{Esphere}, since all
terms are quadratic in $H$ except for the constant term in \eqref{Esphere},
which corresponds to a multiple of the area of $\Si$.  Of course there are
no compact minimal submanifolds of $\R^n$.  But the statement holds also 
for noncompact minimal submanifolds in the sense that the Euler-Lagrange
equation for $\cE$ holds, corresponding to compactly supported variations. 
Note that in the case that $\Si\subset S^n$ is minimal, we have
$\cEb(\Si)=48 \operatorname{Area}(\Si)$.

\subsection{Products of Spheres}\label{prod}
Proposition~\ref{minimalcritical} provides many examples of $\cE$-critical 
submanifolds.  Just as for the classical $k=2$ Willmore energy, minimal 
submanifolds of $S^n$ or their images in $\R^n$ under stereographic
projection are $\cE$-critical.  So a totally geodesic
$S^4\subset S^n$, or any round $S^4\subset \R^n$, is $\cE$-critical.  
Likewise, the standard examples of minimal embeddings of 4-dimensional
products of spheres in $S^n$ are $\cE$-critical.  These are:
\begin{equation}\label{minimalones}
\begin{gathered}
S^2\big(\tfrac{1}{\sqrt{2}}\big)\times
S^2\big(\tfrac{1}{\sqrt{2}}\big)\subset S^5\\ 
S^1\big(\tfrac{1}{2}\big)\times
S^3\big(\tfrac{\sqrt{3}}{2}\big)\subset S^5\\ 
S^1\big(\tfrac{1}{2}\big)\times S^1\big(\tfrac{1}{2}\big)\times 
S^2\big(\tfrac{1}{\sqrt{2}}\big)\subset S^6 \\
S^1\big(\tfrac{1}{2}\big)\times S^1\big(\tfrac{1}{2}\big)\times
S^1\big(\tfrac{1}{2}\big)\times S^1\big(\tfrac{1}{2}\big)\subset S^7
\end{gathered}
\qquad\qquad
\begin{aligned}
\cEb&=192\pi^2  \\
\cEb&=36\sqrt{3}\pi^3\\
\cEb&=96\pi^3\\
\cEb&=48\pi^4
\end{aligned}
\end{equation}

For each of the four topologies appearing in \eqref{minimalones}, there is 
a family of embeddings generalizing \eqref{minimalones} obtained by 
varying the radii of the factor spheres.  Namely, we have 
\begin{equation}\label{families}
\begin{gathered}
S^2(r_1)\times S^2(r_2) \subset S^5  \\
S^1(r_1)\times S^3(r_2)\subset S^5 \\
S^1(r_1)\times S^1(r_2)\times S^2(r_3)\subset S^6 \\
S^1(r_1)\times S^1(r_2)\times S^1(r_3)\times S^1(r_4)\subset S^7, 
\end{gathered}
\end{equation}
where in each case $\sum r_k^2 = 1$.  In this section we calculate
explicitly the energy of these embeddings as a function of the $r_k$.  We
deduce two consequences.  One consequence is that in the three 
families other than $S^2\times S^2\subset S^5$, $\cE$ is already unbounded
above and below when restricted to the family.  This will prove 
Proposition~\ref{no} in these cases.  The other consequence is that we
identify a non-minimal critical point of $\cE$ in each of the three
families other than $S^2\times S^2$.  Specifically, we prove
\begin{proposition}
In addition to \eqref{minimalones}, the following are critical for 
$\cE$: 
\begin{equation}\label{nonminimal}
\begin{gathered}
S^1\big(\sqrt{\tfrac38}\big)\times S^3\big(\sqrt{\tfrac58}\big)\subset S^5\\
S^1\big(\sqrt{\tfrac{5}{24}}\big)\times
S^1\big(\sqrt{\tfrac{9}{24}}\big)\times
S^2\big(\sqrt{\tfrac{10}{24}}\big)\subset S^6 \\
S^1\big(\sqrt{\tfrac{5}{24}}\big)\times
  S^1\big(\sqrt{\tfrac{5}{24}}\big)\times 
S^1\big(\sqrt{\tfrac{5}{24}}\big)\times 
S^1\big(\sqrt{\tfrac{9}{24}}\big)\subset S^7
\end{gathered}
\qquad
\begin{aligned}
\cEb&=16\sqrt{15}\pi^3  \\[.05in]
\cEb&=\tfrac{128\sqrt{5}}{3}\pi^3\\[.05in]
\cEb&=\tfrac{64\sqrt{5}}{3}\pi^4
\end{aligned}
\end{equation}
\end{proposition}

\noindent
Zhang (\cite{Z}) asserts that \eqref{minimalones} together with  
\eqref{nonminimal} constitute all $\cE$-critical embeddings in the families 
\eqref{families}.  

Consider first $S^2(r_1) \times S^2(r_2)\subset S^5$.  Parametrize
$S^2(r_1)\subset \R^3$ using spherical coordinates:
$$
y_1=r_1(\sin\phi_{1}\cos\theta_{1},\sin\phi_{1}\sin\theta_{1},\cos\phi_{1})
$$ 
and likewise $S^2(r_2)\subset \R^3$:
$$
y_2=r_2(\sin\phi_2\cos\theta_2,\sin\phi_2\sin\theta_2,\cos\phi_2)
$$
where $0\leq \theta_1,\theta_2< 2\pi$, $0\leq \phi_1,\phi_2 \leq \pi$.  Then
$x=(y_1,y_2)$ is our embedding $S^2(r_1)\times S^2(r_2)\subset
S^5(1)\subset \R^6$.   

The following is an orthonormal basis for the tangent space to $S^2(r_1)\times S^2(r_2)$:
\begin{align*}
e_{1} =& \frac{x_{\phi_{1}}}{\vert x_{\phi_{1}}\vert} = \frac{1}{r_{1}}x_{\phi_{1}}\\
=& (\cos\phi_{1}\cos\theta_{1},\cos\phi_{1}\sin\theta_{1},-\sin\phi_{1},0,0,0)\\
e_{2} =& \frac{x_{\theta_{1}}}{\vert x_{\theta_{1}}\vert} = \frac{1}{r_{1}\sin\phi_{1}}x_{\theta_{1}}\\
=& (-\sin\theta_{1},\cos\theta_{1},0,0,0,0)\\
e_{3} =& \frac{x_{\phi_{2}}}{\vert x_{\phi_{2}}\vert} = \frac{1}{r_{2}}x_{\phi_{2}}\\
=& (0,0,0,\cos\phi_{2}\cos\theta_{2},\cos\phi_{2}\sin\theta_{2},-\sin\phi_{2})\\
e_{4} =& \frac{x_{\theta_{2}}}{\vert x_{\theta_{2}}\vert} = \frac{1}{r_{2}\sin\phi_{2}}x_{\theta_{2}}\\
=& (0,0,0,-\sin\theta_{2},\cos\theta_{2},0).
\end{align*}
Set 
\begin{equation}\label{normal}
\nu = \big(-\frac{r_{2}}{r_{1}}y_1, \frac{r_{1}}{r_{2}}y_2\big).
\end{equation}
It is easily 
checked that $e_{1},e_{2},e_{3},e_{4},\nu$ is an orthonormal basis for the tangent
space to $S^5$ along $S^{2}(r_{1})\times S^{2}(r_2)$. 
In the subsequent discussion in which a unit normal has been chosen for a
particular hypersurface embedding, we identify $L$ and $H$ with their
scalar counterparts determined by the chosen normal.  
The second fundamental form in this basis is thus given by    
\begin{equation}\label{2ff}
L_{\al\be}=\langle \na_{e_\al}e_\be,\nu\rangle. 
\end{equation}
Now
$$
\na_{e_\al}e_\be=\nabla_{\frac{\partial_{\al} x}{\vert \partial_{\al}x
    \vert}}e_{\be} = \frac{1}{\vert\partial_{\al} x\vert}\partial_{\al}
e_{\be}, 
$$
where $\partial_{1} = \partial_{\phi_{1}}, \partial_{2} =
\partial_{\theta_{1}}, \partial_{3} = \partial_{\phi_{2}}, \partial_{4} =
\partial_{\theta_{2}}$. From the expressions given above for $e_{\al}$, it
is easy to see that $\nabla_{e_{\al}}e_{\be} = 0$ except for the following
cases:  
\begin{align*}
\nabla_{e_{1}}e_{1} =& \frac{1}{r_{1}}(-\sin\phi_{1}\cos\theta_{1},
-\sin\phi_{1}\sin\theta_{1},-\cos\phi_{1},0,0,0)\\ 
\nabla_{e_{2}}e_{1} =&
\frac{1}{r_{1}\sin\phi_{1}}(-\cos\phi_{1}\sin\theta_{1},
\cos\phi_{1}\cos\theta_{1},0,0,0,0)\\ 
\nabla_{e_{2}}e_{2} =&
\frac{1}{r_{1}\sin\phi_{1}}(-\cos\theta_{1},-\sin\theta_{1},0,0,0,0)\\ 
\nabla_{e_{3}}e_{3} =&
\frac{1}{r_{2}}(0,0,0,-\sin\phi_{2}\cos\theta_{2},
-\sin\phi_{2}\sin\theta_{2},-\cos\phi_{2})\\ 
\nabla_{e_{4}}e_{3} =&
\frac{1}{r_{2}\sin\phi_{2}}(0,0,0,-\cos\phi_{2}\sin\theta_{2},
\cos\phi_{2}\cos\theta_{2},0)\\ 
\nabla_{e_{4}}e_{4} =& \frac{1}{r_{2}\sin\phi_{2}}(0,0,0,-\cos\theta_{2},-\sin\theta_{2},0).
\end{align*}
Therefore, in this basis, 
\begin{equation}\label{LS2}
L_{\al\be}=\begin{pmatrix} 
\frac{r_{2}}{r_{1}} & 0 & 0 & 0\\
0 & \frac{r_{2}}{r_{1}} & 0 & 0\\
0 & 0 & -\frac{r_{1}}{r_{2}} & 0\\
0 & 0 & 0 & -\frac{r_{1}}{r_{2}}
\end{pmatrix}.
\end{equation}
Thus 
$$
H=2\Big(\frac{r_{2}}{r_{1}}-\frac{r_{1}}{r_{2}}\Big),\qquad 
|L|^2=2\Big(\frac{r_1^2}{r_2^2}+\frac{r_2^2}{r_1^2}\Big).
$$
Substituting into \eqref{Esphere} and simplifying using $r_1^2+r_2^2=1$
gives 
\begin{equation}\label{tform}
\begin{split}
\cEb&=-16\pi^2\, \frac{r_1^4+r_2^4-14r_1^2r_2^2}{r_1^2r_2^2}\\
&=-16\pi^2\Big(t^2+t^{-2}-14\Big)
\end{split}
\end{equation}
with $\frac{r_1}{r_2}=t$.  The function $t^2+t^{-2}$ 
of $t\in (0,\infty)$ has a unique critical point at $t=1$, a global
minimum.  Thus $\cEb$ is 
unbounded below on this family, with a maximum of $192\pi^2$ achieved at
the minimal embedding.  We remark that since the formula for $\cEb$ above
is homogeneous of degree $0$ in $(r_1,r_2)$, it is valid by conformal
invariance for the embedding 
$S^2(r_1)\times S^2(r_2)\subset S^5(\sqrt{r_1^2+r_2^2})$ for all $r_1$,
$r_2>0$.  

The calculations for the other families in \eqref{families} are similar.
We briefly outline the computations in each case.   

For $S^1\times S^3$, parametrize $S^1(r_1)\subset \R^2$ as  
$$
y_1=r_{1}(\cos\theta_1,\sin\theta_1)
$$ 
and $S^3(r_2)\subset \R^4$ as 
$$
y_2=r_{2}(\sin\phi_{1}\sin\phi_{2}\sin\theta_{2},\sin\phi_{1}\sin\phi_{2}\cos\theta_{2},
\sin\phi_{1}\cos\phi_{2},\cos\phi_{1}).
$$
Then $x=(y_1,y_2)$ is our embedding $S^1(r_1)\times S^3(r_2)\subset
S^5(1)\subset \R^6$.  The vectors 
$$
e_{1} = \frac{x_{\theta_1}}{\vert x_{\theta_1}\vert},\quad
e_{2} = \frac{x_{\phi_{1}}}{\vert x_{\phi_{1}}\vert},\quad
e_{3} = \frac{x_{\phi_{2}}}{\vert x_{\phi_{2}}\vert},\quad
e_{4} = \frac{x_{\theta_{2}}}{\vert x_{\theta_{2}}\vert},\quad 
\nu = \big(-\frac{r_{2}}{r_{1}}y_1, \frac{r_{1}}{r_{2}}y_2\big) 
$$
form an orthonormal basis for the tangent
space to $S^5$ along $S^{1}(r_{1})\times S^{3}(r_2)$. 
The second fundamental form \eqref{2ff} in this basis now takes
the form
\begin{equation*}
L_{\al\be}=\begin{pmatrix} 
\frac{r_{2}}{r_{1}} & 0 & 0 & 0\\
0 & -\frac{r_{1}}{r_{2}} & 0 & 0\\
0 & 0 & -\frac{r_{1}}{r_{2}} & 0\\
0 & 0 & 0 & -\frac{r_{1}}{r_{2}}
\end{pmatrix}.
\end{equation*}

\noindent
So 
$$
H=\frac{r_{2}}{r_{1}}-3\frac{r_{1}}{r_{2}},\qquad 
|L|^2=\frac{r_2^2}{r_1^2}+3\frac{r_1^2}{r_2^2}.
$$
This time substituting into \eqref{Esphere} and simplifying using
$r_1^2+r_2^2=1$ gives  
\[
\begin{split}
\cEb&=\frac{9\pi^3}{4}\,
\frac{15r_1^4r_2^2+14r_1^2r_2^4-r_2^6}{r_1^3r_2^3} \\
&=\frac{9\pi^3}{4}\,\frac{15t^4+14t^2-1}{t^3}
\end{split}
\]
with $t=\frac{r_1}{r_2}$.   
This function of $t$ goes to $+\infty$ as $t\rightarrow \infty$ and goes to
$-\infty$ as $t\rightarrow 0$.  So it is unbounded above and below.  It has
two critical points, a local maximum at $t=\frac{1}{\sqrt{3}}$,
corresponding to the minimal embedding, and a local minimum at
$t=\sqrt{\frac{3}{5}}$.  The Principle of 
Symmetric Criticality (\cite{P}), or alternately Hsiang's argument
\cite{H} concerning critical orbits of compact groups of isometries, which
applies equally well to $\cE$ as to the area functional, implies   
that $S^1\big(\sqrt{\frac38}\big)\times S^3\big(\sqrt{\frac58}\big)\subset
S^5$ is critical for $\cE$.  This is the first example in
\eqref{nonminimal}.

For $S^1\times S^1\times S^2$, parametrize $S^1(r_1)\subset \R^2$ and 
$S^1(r_2)\subset \R^2$ by
$$
y_1=r_{1}(\cos\theta_1,\sin\theta_1),\qquad
y_2=r_{2}(\cos\theta_2,\sin\theta_2) 
$$ 
and $S^2(r_3)\subset \R^3$ by
$$
y_3=r_3(\sin\phi\cos\theta_{3},\sin\phi\sin\theta_{3},\cos\phi).
$$ 
Then $x=(y_1,y_2,y_3)$ is our embedding $S^1(r_1)\times S^1(r_2)\times 
S^2(r_3)\subset S^6(1)\subset \R^7$.  
The vectors 
$$
e_{1} = \frac{x_{\theta_1}}{\vert x_{\theta_1}\vert},\quad
e_{2} = \frac{x_{\theta_2}}{\vert x_{\theta_2}\vert},\quad
e_{3} = \frac{x_{\theta_3}}{\vert x_{\theta_3}\vert},\quad
e_{4} = \frac{x_{\phi}}{\vert x_{\phi}\vert}
$$
form an orthonormal basis for the tangent space to $S^1(r_1)\times 
S^1(r_2)\times  S^2(r_3)$.  The vectors
\begin{align*}
\nu_{1} =& \Big(-\frac{r_{2}}{r_{1}\sqrt{r_{1}^{2}+r_{2}^{2}}}y_{1},
\frac{r_{1}}{r_{2}\sqrt{r_{1}^{2}+r_{2}^{2}}}y_{2},0\Big)\\ 
\nu_{2} =& \Big(\frac{r_{3}}{\sqrt{r_{1}^{2}+r_{2}^{2}}}y_{1},
\frac{r_{3}}{\sqrt{r_{1}^{2}+r_{2}^{2}}}y_{2},-\frac{\sqrt{r_{1}^{2}+r_{2}^{2}}}{r_{3}}y_3\Big)\\ 
\end{align*}
form an orthonormal basis for the normal space.  The second fundamental
forms $L_{\al\be}^{\al'}=\langle \na_{e_\al}e_\be,\nu_{\al'}\rangle$, $\al'=1,2$, are given
by 
\begin{equation*}
L^1_{\al\be}=\begin{pmatrix}
\frac{r_{2}}{r_{1}\sqrt{r_1^2+r_2^2}} & 0 & 0 & 0\\
0 & -\frac{r_{1}}{r_{2}\sqrt{r_1^2+r_2^2}} & 0 & 0\\
0 & 0 & 0 & 0\\
0 & 0 & 0 & 0
\end{pmatrix},\,\,\,\,
L^2_{\al\be}=\begin{pmatrix}
-\frac{r_{3}}{\sqrt{r_1^2+r_2^2}} & 0 & 0 & 0\\
0 & -\frac{r_{3}}{\sqrt{r_1^2+r_2^2}} & 0 & 0\\
0 & 0 & \frac{\sqrt{r_1^2+r_2^2}}{r_{3}} & 0\\
0 & 0 & 0 & \frac{\sqrt{r_1^2+r_2^2}}{r_{3}}
\end{pmatrix}
\end{equation*}
and the mean curvatures by
$$
H^1=\frac{r_2^2-r_1^2}{r_1r_2\sqrt{r_1^2+r_2^2}},\qquad
H^2=\frac{2(r_1^2+r_2^2-r_3^2)}{r_3\sqrt{r_1^2+r_2^2}}.
$$
We used {\it Mathematica} and $r_1^2+r_2^2+r_3^2=1$ to calculate that    
\[
\begin{split}
\cEb&=-\frac{\pi^3}{r_1^3r_2^3r_3^3}
\,\Big[16r_1^4r_2^4r_3 
-56(r_1^4r_2^2+r_2^4r_1^2)r_3^3+\big(9(r_1^4+r_2^4)-14r_1^2r_2^2\big)r_3^5\Big]\\
&=-\frac{\pi^3}{t_1^3t_2^3}\,\Big[16t_1^4t_2^4 
-56(t_1^4t_2^2+t_2^4t_1^2)+9(t_1^4+t_2^4)-14t_1^2t_2^2\Big]
\end{split}
\]
with $t_1=\frac{r_1}{r_3}$, $t_2=\frac{r_2}{r_3}$.  
It is easily seen that this function of $(t_1,t_2)\in (0,\infty)\times 
(0,\infty)$ is unbounded above and below (for instance, this is already the
case when restricted to $t_2=1$) and has critical points at 
$(t_1,t_2)=\big(\frac{1}{\sqrt{2}},\frac{1}{\sqrt{2}}\big)$ and 
$(t_1,t_2)=\big(\sqrt{\frac{5}{10}},\sqrt{\frac{9}{10}}\big)$.  (By  
symmetry, another critical point is obtained from the latter by
interchanging $t_1$ and $t_2$.)  The first critical point corresponds to
the minimal embedding $r_1=r_2=\frac12$, $r_3=\frac{1}{\sqrt{2}}$.  The second
corresponds to the non-minimal $\cE$-critical embedding
$S^1\big(\sqrt{\frac{5}{24}}\big)\times
  S^1\big(\sqrt{\frac{9}{24}}\big)\times 
S^2\big(\sqrt{\frac{10}{24}}\big)\subset S^6$.   This
is the second example in \eqref{nonminimal}.  

For $(S^1)^4$, parametrize $S^1(r_\al)\subset \R^2$, $1\leq \al\leq 4$,  
by
$$
y_\al=r_{\al}(\cos\theta_\al,\sin\theta_\al).
$$ 
Then $x=(y_1,y_2,y_3,y_4)$ is our embedding $S^1(r_1)\times S^1(r_2)\times   
S^1(r_3)\times S^1(r_4)\subset S^7(1)\subset \R^8$.  
The vectors 
$$
e_{\al} = \frac{x_{\theta_\al}}{r_\al},\qquad 1\leq \al\leq 4,
$$
form an orthonormal basis for the tangent space to $S^1(r_1)\times 
S^1(r_2)\times  S^1(r_3)\times S^1(r_4)$.  The vectors
\begin{align*}
\nu_{1} =& \Big(-\frac{r_{2}}{r_{1}}y_{1},\frac{r_{1}}{r_{2}}y_{2},
-\frac{r_{4}}{r_{3}}y_{3},\frac{r_{3}}{r_{4}}y_{4}\Big)\\ 
\nu_{2} =& \Big(-\frac{r_{3}}{r_{1}}y_{1},\frac{r_{4}}{r_{2}}y_{2},
\frac{r_{1}}{r_{3}}y_{3},-\frac{r_{2}}{r_{4}}y_{4}\Big)\\
\nu_{3} =& \Big(-\frac{r_{4}}{r_{1}}y_{1},-\frac{r_{3}}{r_{2}}y_{2},
\frac{r_{2}}{r_{3}}y_{3},\frac{r_{1}}{r_{4}}y_{4}\Big) 
\end{align*}
form an orthonormal basis for the normal space.  The second fundamental
forms $L_{\al\be}^{\al'}=\langle \na_{e_\al}e_\be,\nu_{\al'}\rangle$ are
given by  
\begin{equation*}
L^1_{\al\be}=\begin{pmatrix}
\frac{r_{2}}{r_{1}} & 0 & 0 & 0\\
0 & -\frac{r_{1}}{r_{2}} & 0 & 0\\
0 & 0 & \frac{r_{4}}{r_{3}} & 0\\
0 & 0 & 0 & -\frac{r_{3}}{r_{4}}
\end{pmatrix},\,
L^2_{\al\be}=\begin{pmatrix}
\frac{r_{3}}{r_{1}} & 0 & 0 & 0\\
0 & -\frac{r_{4}}{r_{2}} & 0 & 0\\
0 & 0 & -\frac{r_{1}}{r_{3}} & 0\\
0 & 0 & 0 & \frac{r_{2}}{r_{4}}
\end{pmatrix},\,
L^3_{\al\be}=\begin{pmatrix}
\frac{r_{4}}{r_{1}} & 0 & 0 & 0\\
0 & \frac{r_{3}}{r_{2}} & 0 & 0\\
0 & 0 & -\frac{r_{2}}{r_{3}} & 0\\
0 & 0 & 0 & -\frac{r_{1}}{r_{4}}
\end{pmatrix}
\end{equation*}
and the mean curvatures by
$$
H^1=\frac{r_2^2-r_1^2}{r_1r_2}+\frac{r_4^2-r_3^2}{r_3r_4},\qquad
H^2=\frac{r_3^2-r_1^2}{r_1r_3}+\frac{r_2^2-r_4^2}{r_2r_4},\qquad
H^3=\frac{r_4^2-r_1^2}{r_1r_4}+\frac{r_3^2-r_2^2}{r_2r_3}.
$$
We used {\it Mathematica} and $\sum_{\al=1}^4r_\al^2=1$ to calculate that   
\[
\begin{split}
\cEb&=-\frac{\pi^4}{r_1^3r_2^3r_3^3r_4^3}\,\Big[36\Sym(r_1^4r_2^4r_3^4) 
-84\Sym(r_1^4r_2^4r_3^2r_4^2)\Big]\\
&=-\frac{\pi^4}{t_1^3t_2^3t_3^3}
\Big[9t_1^4t_2^4t_3^4+27\Sym(t_1^4t_2^4)
-42\Sym(t_1^4t_2^4t_3^2)-42\Sym(t_1^4t_2^2t_3^2)\Big], 
\end{split}
\]
where $t_k=\frac{r_k}{r_4}$, $1\leq k\leq 3$.  Here we write
\begin{align*}
\Sym(r_1^ar_2^br_3^cr_4^d)&=\frac{1}{4!}\sum_{\sigma\in
  S_4}r_{\si(1)}^ar_{\si(2)}^br_{\si(3)}^cr_{\si(4)}^d \\
\Sym(t_1^at_2^bt_3^c)&=\frac{1}{3!}\sum_{\sigma\in
  S_3}t_{\si(1)}^at_{\si(2)}^bt_{\si(3)}^c
\end{align*}
for the symmetrization of monomials.  
It is easily seen that $\cEb$, viewed as a function of $(t_1,t_2,t_3)\in
(0,\infty)^3$, is 
unbounded above and below (for instance, this is already the case when
restricted to $t_2=t_3=1$) and has critical points at 
$(t_1,t_2,t_3)=(1,1,1)$ and 
$(t_1,t_2,t_3)=\big(\sqrt{\frac59},\sqrt{\frac59},\sqrt{\frac59}\big)$.
The first critical point corresponds to the minimal embedding 
$r_1=r_2=r_3=r_4=\frac12$.  The second corresponds to the non-minimal
$\cE$-critical embedding 
$S^1\big(\sqrt{\tfrac{5}{24}}\big)\times
  S^1\big(\sqrt{\tfrac{5}{24}}\big)\times 
S^1\big(\sqrt{\tfrac{5}{24}}\big)\times 
S^1\big(\sqrt{\tfrac{9}{24}}\big)\subset S^7$.  This is the third
example in \eqref{nonminimal}. 

We close our discussion of the energy of the embeddings \eqref{families}
with an observation that we find truly remarkable.  Namely, the energy and  
critical points for the first three of the families in \eqref{families}, as
well as for equatorial $S^4\subset S^5$, can all be derived from the energy 
and critical points for the 4-torus family $S^1(r_1)\times S^1(r_2)\times 
S^1(r_3)\times S^1(r_4)\subset S^7$ by specializing the radii.  Precisely,
the following relations hold:  
\begin{align*}
\mathcal{E}(S^{1}(r_{1})\times S^{1}(r_{2})\times S^{1}(r_{3})\times
S^{1}(r_{3}))  
=&
\frac{\pi}{2}\,\mathcal{E}(S^{1}(r_{1})
\times S^{1}(r_{2})\times S^{2}(\sqrt{2}r_{3}))\\ 
\mathcal{E}(S^{1}(r_{1})\times S^{1}(r_{2})
\times S^{1}(r_{2})\times S^{1}(r_{2}))
=&
\frac{4\pi}{3\sqrt{3}}\,\mathcal{E}(S^{1}(r_{1})
\times S^{3}(\sqrt{3}r_{2})) \\   
\mathcal{E}(S^{1}(r_{1})\times S^{1}(r_{1})
\times S^{1}(r_{2})\times S^{1}(r_{2})) 
=&
\frac{\pi^{2}}{4}\,\mathcal{E}(S^{2}(\sqrt{2}r_{1})
\times S^{2}(\sqrt{2}r_{2}))\\  
\mathcal{E}(S^{1}(r_{1})\times S^{1}(r_{1})
\times S^{1}(r_{1})\times S^{1}(r_{1})) 
=&
\frac{3\pi^{2}}{8}\,\mathcal{E}(S^{4}(2r_{1})).  
\end{align*}
Moreover, the constants appearing in these relations are the 
corresponding ratios of the areas of the factor spheres.  Namely,
\begin{align*}
\frac{\pi}{2}&=
\frac{A(S^1(r_1))A(S^1(r_2))A(S^1(r_3))^2}{A(S^1(r_1))A(S^1(r_2))A(S^2(\sqrt{2}r_3))}\\
\frac{4\pi}{3\sqrt{3}}&=
\frac{A(S^1(r_1))A(S^1(r_2))^3}{A(S^1(r_1))A(S^3(\sqrt{3}r_2))}\\
\frac{\pi^2}{4}&=
\frac{A(S^{1}(r_{1}))^2A(S^{1}(r_{2}))^2}
{A(S^{2}(\sqrt{2}r_{1}))A(S^{2}(\sqrt{2}r_{2}))}\\ 
\frac{3\pi^2}{8}&=
\frac{A(S^{1}(r_{1}))^4}{A(S^4(2r_1))}.
\end{align*}
It is clear that these relations enable the deduction of the critical
points of any of the families from those for $S^{1}(r_{1})\times
S^{1}(r_{2})\times S^{1}(r_{3})\times S^{1}(r_{4})$.  Also, one may deduce
direct relations between the energy of any one of the families and any
other family with fewer factors.  That is, the energy of the family 
$S^{1}(r_{1})\times S^{1}(r_{2})\times S^{2}(r_{3})$ determines the energy
of the $S^{1}\times S^{3}$, $S^{2}\times S^{2}$, and $S^{4}$ families, and
the energy of either family $S^{1}\times S^{3}$ or $S^{2}\times S^{2}$
determines that of the $S^{4}$ family, by similar relations.  

We also note that the same sort of relation holds for the usual
2-dimensional Willmore energy, with the same rule for the constant.
Namely, for the 2-dimensional energy $\cE$ given in
Corollary~\ref{Eformcor}, we have  
$$
\mathcal{E}(S^{1}(r_{1})\times S^{1}(r_{1}))
=
\frac{\pi}{2}\,\mathcal{E}(S^{2}(\sqrt{2}r_{1}))
$$
with $\frac{\pi}{2} = \frac{A(S^1(r_1))^2}{A(S^2(\sqrt{2}r_1))}$.  
(No such relation seems to hold for the Willmore energy in the form 
$\int_{\Si}|\mathring{L}|^2$ owing to the extra term involving
the Euler characteristic.)  
We have no understanding of why these relations should be true other than
verifying them from the formulas.  It would be interesting to provide a   
geometric explanation.

\subsection{Anchor Rings}\label{anchor}
We first indicate how the embeddings $S^j\times S^k\subset S^{j+k+1}$
considered in \S\ref{prod} can be composed with stereographic projection to
obtain embeddings $S^j\times S^k\subset \R^{j+k+1}$ which are higher
dimensional versions of anchor rings.
We then construct a family of embeddings of $S^2\times S^2$ in $\R^5$ for
which $\cE\rightarrow \infty$ by dilating such an anchor ring in only some
of the variables.  This combined with the results in \S\ref{prod} enables
us to prove Proposition~\ref{no}.  

Fix $R>0$ and $0<r<R$.  Define the anchor ring embedding 
$S^j\times S^k\rightarrow \R^{j+k+1}$ by
\begin{equation}\label{anchorparam}
S^j\times S^k\ni (y,z)\rightarrow \big((R+rw)y,rv\big),
\end{equation}
where we have written $z=(v,w)\in S^k\subset \R^{k+1}$ with $v\in \R^k$,
$w\in \R$.  
The image $T^{j,k}_{R,r}$ is the tube of radius $r$ about $S^j(R)\subset 
\R^{j+1}\times \{0\}$.  Note that $j$ and $k$ are not treated
symmetrically:  for $j\neq k$ the two embeddings $T^{j,k}_{R,r}$ and
$T^{k,j}_{R,r}$ of $S^j\times S^k$ into $\R^{j+k+1}$ are distinct (even 
disregarding the interchange of the factors).    

Next let $r_1^2+r_2^2=1$ and recall the product embedding $S^j(r_1)\times 
S^k(r_2)\subset S^{j+k+1}\subset \R^{j+1}\times \R^{k+1}$.  Consider 
its image under stereographic projection, where the 
base point for the stereographic projection is a point of $S^{j+k+1}$ lying
in $\{0\}\times \R^{k+1}$, which we can take to be $(0,\ldots,0,1)$ by a
rotation of $\R^{k+1}$.  Thus the stereographic projection
$\pi:S^{j+k+1}\rightarrow \R^{j+k+1}$ is  
$$
\pi(x',x_{j+k+2})=\frac{x'}{1-x_{j+k+2}},\qquad\qquad
x'=(x_1,\ldots,x_{j+k+1}). 
$$
It is an easy verification which we leave to the reader that 
$\pi\big(S^j(r_1)\times S^k(r_2)\big)=T^{j,k}_{R,r}$ with $R=1/r_1$ and 
$r=r_2/r_1$.  Since $\pi$ is conformal, 
$\cE(S^j(r_1)\times S^k(r_2))=\cE(T^{j,k}_{R,r})$.  Thus the $\cE$-critical 
embeddings of $S^2\times S^2$ and $S^1\times S^3$ in $S^5$ discussed above
give $\cE$-critical ``anchor ring'' tubes in $\R^5$.   
For $S^2\times S^2$, we obtain the single critical anchor ring tube 
$T^{2,2}_{\sqrt{2},1}$, which is the stereographic image of the 
minimal $S^2\big(\tfrac{1}{\sqrt{2}}\big)\times 
S^2\big(\tfrac{1}{\sqrt{2}}\big)\subset S^5$.
For $S^1\times S^3$, 
we obtain the two critical anchor ring tubes $T^{1,3}_{2,\sqrt{3}}$ and 
$T^{3,1}_{2/\sqrt{3},1/\sqrt{3}}$ corresponding to the minimal embedding,
and $T^{1,3}_{\sqrt{8/3},\sqrt{5/3}}$ and 
$T^{3,1}_{\sqrt{8/5},\sqrt{3/5}}$ corresponding to the non-minimal
embedding.  Note that $T^{j,k}_{R,r}$ is a dilate of $T^{j,k}_{cR,cr}$ for 
any $c>0$, so we conclude from the last three embeddings that 
$T^{3,1}_{2,1}$, $T^{1,3}_{\sqrt{8},\sqrt{5}}$, and 
$T^{3,1}_{\sqrt{8},\sqrt{3}}$ are $\cE$-critical.  

We have written down generalizations of the above embeddings to
``generalized anchor 
rings'' which arise as the images under stereographic projection of 
$S^{j_1}(r_1)\times \cdots S^{j_l}(r_l)\subset S^{j_1+\ldots+j_l
  +l-1}\subset \R^{j_1+\ldots +j_l+l}$ for $\sum_{i=1}^lr_i^2=1$.
Specializing to the $\cE$-critical products $S^1(r_1)\times S^1(r_2)\times
S^2(r_3)\subset S^6$ and $S^1(r_1)\times S^1(r_2)\times S^1(r_3)\times
S^1(r_4)\subset S^7$
discussed above gives $\cE$-critical generalized anchor ring embeddings in
$\R^6$ and $\R^7$.  

We next exhibit a family of embeddings of $S^2\times S^2$ in $\R^5$ with
energy unbounded above.  For $a>0$, let $\delta_a:\R^5\rightarrow 
\R^5$ be given by $\delta_a(y,v)=(y,av)$ for $y\in \R^3$, $v\in \R^2$.   

\begin{proposition}\label{a}
$$
\cE\big(\delta_a(T^{2,2}_{\sqrt{2},1})\big)=\frac{2\pi^2}{35}a^4 +o(a^4)  
$$
as $a\rightarrow \infty$.
\end{proposition}
\begin{proof}
Conformal invariance implies
$\cE\big(\delta_a(T_{\sqrt{2},1})\big)=\cE\big(\delta_a(T_{1,1/\sqrt{2}})\big)$. 
(We suppress the ${}^{2,2}$ on $T^{2,2}_{R,r}$ throughout this proof.)   
Consider $\delta_a(T_{1,r})$ for $0<r<1$; we will 
set $r=1/\sqrt{2}$ later.  

Parametrize $\delta_a(T_{1,r})$ by introducing 
spherical coordinates for $y$, $z$ in \eqref{anchorparam}:
$$
x=\big((1+r\cos\phi_2)y, arv\big)
$$
with
\begin{align*}
y&=(\sin\phi_1\cos\theta_1,\sin\phi_1\sin\theta_1,\cos\phi_1)\\
v&=\sin\phi_2(\cos\theta_2,\sin\theta_2),
\end{align*}
where $0\leq \theta_1,\theta_2< 2\pi$, $0\leq \phi_1,\phi_2 \leq \pi$.   
The tangent vectors $e_1=x_{\phi_1}$, $e_2=x_{\theta_1}$, $e_3=x_{\phi_2}$,
$e_4=x_{\theta_2}$ are orthogonal with 
\begin{equation}\label{xderivs}
\begin{split}
|x_{\phi_1}|^2&=(1+r\cos\phi_2)^2 \\
|x_{\theta_1}|^2&=(1+r\cos\phi_2)^2\sin^2\phi_1\\
|x_{\phi_2}|^2&=r^2(a^2\cos^2\phi_2+\sin^2\phi_2) \\ 
|x_{\theta_2}|^2&=a^2r^2\sin^2\phi_2.
\end{split}
\end{equation}
A unit normal is
$$
\nu=\frac{1}{\ell(\phi_2)}\big(a(\cos\phi_2)y,v\big)
$$
where $\ell(\phi_2)=\sqrt{a^2\cos^2\phi_2+\sin^2\phi_2}
=\sqrt{1+(a^2-1)\cos^2\phi_2}$.  
The second fundamental form \eqref{2ff} in this basis is 
\begin{equation*}
L_{\al\be}=-\frac{a}{\ell(\phi_2)} 
\begin{pmatrix} 
\cos\phi_2(1+r\cos\phi_2) & 0 & 0 & 0\\
0 & \cos\phi_2(1+r\cos\phi_2)\sin^2\phi_1 & 0 & 0\\  
0 & 0 & r & 0\\
0 & 0 & 0 & r\sin^2\phi_2
\end{pmatrix}.
\end{equation*}
Recalling \eqref{xderivs}, contraction gives:
\[
\begin{split}
H&=-\frac{a}{\ell(\phi_2)}\left[\frac{2\cos\phi_2}{1+r\cos\phi_2} 
+\frac{1}{r\ell(\phi_2)^2}+\frac{1}{ra^2}\right]\\
|L|^2&=\frac{a^2}{\ell(\phi_2)^2}\left[\frac{2\cos^2\phi_2}{(1+r\cos\phi_2)^2}  
+\frac{1}{r^2\ell(\phi_2)^4}+\frac{1}{r^2a^4}\right].
\end{split}
\]
Since $H$ depends only on $\phi_2$, it follows that 
$$
|\nabla H|^2 = r^{-2}\ell(\phi_2)^{-2}\big(\pa_{\phi_2}H\big)^2.
$$
This information is sufficient to calculate the integrand 
$|\nabla H|^2 -|L|^2H^2 +\frac{7}{16} H^4$ of
$\cEb\big(\delta_a(T_{1,r})\big)$, and clearly it depends only on
$\phi_2$. Now \eqref{xderivs} also gives  
$$
da=ar^2(1+r\cos\phi_2)^2\ell(\phi_2)\sin\phi_2\sin\phi_1d\phi_1d\theta_1d\phi_2d\theta_2.  
$$
It follows that we can write
$$
\Big(|\nabla H|^2 -|L|^2H^2 +\frac{7}{16} H^4\Big)da 
=I(\cos\phi_2)\sin\phi_1\sin\phi_2d\phi_1d\theta_1d\phi_2d\theta_2
$$
where $I(\cos\phi_2)$ is a function only of $\cos\phi_2$ (depending on $a$
and $r$).   
Upon making the substitution $s=\cos\phi_2$, it follows that
\begin{equation}\label{int2}
\cEb\big(\delta_a(T_{1,r})\big)=8\pi^2\int_{-1}^1 I(s)\,ds.
\end{equation}
We used {\it Mathematica} to calculate $I(s)$ for $r=1/\sqrt{2}$.  The
result is 
\begin{equation}\label{I}
I(s)=-\frac{1}{8a^3}\frac{p(a,s)}
{\left(1+\frac{s}{\sqrt{2}}\right)^2\big(1+(a^2-1)s^2\big)^{11/2}},  
\end{equation}
where $p(a,s)$ is the polynomial in $a$, $s$ given by:
\begin{equation}\label{longone}
\begin{split}
p(a,s)&=
4 s^{12} a^{16}\\
&+\Big(-24 s^{12} -24 \sqrt{2} s^{11} -24  s^{10} -24 \sqrt{2} s^9\Big)
a^{14}\\ 
&+\Big(54 s^{12}+92 \sqrt{2} s^{11} -16 s^{10} -128 \sqrt{2} s^9 -218 s^8
-172 \sqrt{2} s^7 -132 s^6\Big) a^{12}\\
&+\Big(-50 s^{12} 
-106 \sqrt{2} s^{11} 
+226 s^{10} +502 \sqrt{2} s^9 +246 s^8 -162 \sqrt{2} s^7 -234 s^6
-286 \sqrt{2} s^5\\
&\qquad
 -380 s^4 -188 \sqrt{2} s^3 -144 s^2\Big)
a^{10}\\ 
&+\Big(\frac{9 s^{12} }{4}-7 \sqrt{2} s^{11}
-354 s^{10} 
-440 \sqrt{2} s^9 +\frac{541 s^8 }{2}+632 \sqrt{2} s^7 
+25 s^6 -350 \sqrt{2} s^5 \\
&\qquad
-\frac{847 s^4 }{4}+95 \sqrt{2} s^3+327 s^2 +118 \sqrt{2} s+9\Big) a^8\\  
&+\Big(27 s^{12} +88 \sqrt{2} s^{11} +208 s^{10} -34 \sqrt{2} s^9
-478 s^8-42 \sqrt{2} s^7 
+816 s^6 +398 \sqrt{2} s^5 \\
&\qquad
 -241 s^4-302 \sqrt{2} s^3 -304 s^2 -108 \sqrt{2} s-28\Big) a^6\\ 
&+\Big(-\frac{25 s^{12} }{2}-38 \sqrt{2} s^{11} -16 s^{10} +146 \sqrt{2}
 s^9 +165 s^8 -322 \sqrt{2} s^7  -570 s^6 +254 \sqrt{2} s^5 \\
&\qquad
+\frac{1523s^4}{2}+64\sqrt{2} s^3 -318 s^2 -104 \sqrt{2} s -10\Big)a^4\\  
&+\Big(-3 s^{12} -14 \sqrt{2} s^{11} -42 s^{10} -4 \sqrt{2} s^9 +100 s^8
+84 \sqrt{2} s^7 -22 s^6 -88 \sqrt{2} s^5 \\ 
&\qquad
 -69 s^4 +10 \sqrt{2} s^3 +32 s^2 +12 \sqrt{2} s +4\Big) a^2\\ 
&+\Big(\frac{9 s^{12}}{4}+9 \sqrt{2}s^{11}
+18 s^{10}-18 \sqrt{2} s^9-\frac{171s^8}{2}-18 \sqrt{2} s^7
+117 s^6+72 \sqrt{2} s^5\\
&\qquad
 -\frac{207 s^4}{4}-63 \sqrt{2} s^3-9 s^2 +18 \sqrt{2} s+9\Big).
\end{split}
\end{equation}
Observe that $p(a,s)$ has degree $16$ in $a$ and
any monomial $s^ia^j$ occurring in $p$ has $j-i\leq 8$.   Moreover, the
only terms with $j-i=8$ are $9a^8-144s^2a^{10}$.  

In order to
derive the asymptotics of $\cEb\big(\delta_a(T_{1,1/\sqrt{2}})\big)$  
as $a\rightarrow \infty$, we have the following lemma.
Set 
$$
\cI_{i,j}(a)=\int_{-1}^{1}\frac{a^{j}s^{i}\,ds}{\left(1+\frac{s}{\sqrt{2}}\right)^2  
    \big(1+(a^2-1)s^2\big)^{11/2}}
$$
\begin{lemma}\label{integralasymptotics} 
$$
\lim_{a\to\infty}\cI_{i,j}(a) = \begin{cases}  
\int_{-\infty}^{\infty}\frac{t^{i}\,dt}{\left(t^2+1\right)^{11/2}} \qquad \qquad
j - i = 1, i < 10\\ 0 \qquad \qquad \qquad \qquad\quad\,\, j - i < 1, j < 11 
\end{cases}
$$
\end{lemma}
\begin{proof}
Make the substitution $t = as$ to rewrite 
\begin{equation*}
\cI_{i,j}(a)=a^{j-i-1}\int_{-a}^{a}{\frac{t^{i}\,dt}
{\left(1+\frac{t}{a\sqrt{2}}\right)^2  
    \big(1+t^2-(t/a)^2\big)^{11/2}}}.  
\end{equation*}
If $j-i=1$ and $i < 10$, dominated convergence shows that
$\lim_{a\rightarrow \infty}\cI_{i,j}(a)=
\int_{-\infty}^{\infty}{\frac{t^{i}\,dt}{\left(t^2+1\right)^{11/2}}}$.  
If $j-i<1$ and $i<10$, the same argument shows that the limit is 
0.   If $j - i < 1$, $j < 11$, and $i\geq 10$, choose $\epsilon$ with 
$0<\epsilon < 11-j$. Then
$$
\Bigg|\int_{-a}^{a}{\frac{a^{j-i-1}t^{i}\,dt}{\left(1+\frac{t}{a\sqrt{2}}\right)^2
      \big(1+t^2-(t/a)^2\big)^{11/2}}}\Bigg|
\leq \int_{-a}^{a}\frac{a^{j-i-1}|t|^{10-\ep}|t|^{i+\ep-10}\,dt}
{\left(1+\frac{t}{a\sqrt{2}}\right)^2
  \big(1+t^2-(t/a)^2\big)^{11/2}}. 
$$
Since $|t| \leq a$ and $i+\ep\geq 10$, we have 
$$
a^{j-i-1}|t|^{10-\ep}|t|^{i+\ep-10}\leq
a^{j-i-1}|t|^{10-\ep}a^{i+\ep-10}=a^{j+\ep-11}|t|^{10-\ep}.
$$
Thus dominated convergence again shows that the limit is 0.   
\end{proof}

According to \eqref{int2}, \eqref{I}, \eqref{longone}, 
$a^{-4}\cEb\big(\delta_a(T_{1,1/\sqrt{2}})\big)$ is a 
linear combination of integrals of the form $\cI_{i,j}(a)$ with $j\leq 9$
and $j-i\leq 1$.  Lemma~\ref{integralasymptotics} shows that the limit
vanishes for all terms with $j-i<1$.  As indicated above, this is all but
two terms.  Thus 
$$
\lim_{a\rightarrow \infty}a^{-4}\cEb\big(\delta_a(T_{1,1/\sqrt{2}})\big) 
=-\pi^2\lim_{a\rightarrow
  \infty}\big(9\,\cI_{0,1}(a)-144\,\cI_{2,3}(a)\big)
=\frac{256\pi^2}{35}.  
$$
The result follows upon recalling that $\cEb=128\cE$.  
\end{proof}

\noindent
{\it Proof of Proposition~\ref{no}}.  
It was shown in \S\ref{prod} that in the three cases other 
than $S^2\times S^2\subset S^5$, $\cE$ is already unbounded above and below   
when restricted to the families considered there.  It was also shown in
\S\ref{prod} that $\cE$ is unbounded below when restricted to the
corresponding family for $S^2\times S^2\subset S^5$.  Proposition~\ref{a}
shows that $\cE$ is unbounded above over embeddings 
$S^2\times S^2\subset S^5$.
\stopthm

\subsection{Second Variation in $S^5$.}\label{2var}
In this section we calculate the second variation of $\cE$ at a minimal
immersed hypersurface $\Si$ in $S^5$.  We specialize the general formula to
$\Si=S^4$ and $\Si=S^2(1/\sqrt{2})\times S^2(1/\sqrt{2})$.  

Let $f:\Si\rightarrow S^5$ be a minimal immersion of $\Si^4$ in $S^5$, and 
$F:\Si\times (-\ep,\ep)$ be a variation of $f$, i.e. $F_0=f$.  Let 
$V=\pa_t F_t|_{t=0}$ denote the variational vector field.  We assume
throughout that $V$ is normal.   
Recall the first and second variations of area:  
\begin{align*}
\pa_tA\big(F_t(\Si)\big)|_{t=0}&=-\int_\Si\langle H,V\rangle\,da_\Si\\
\nabla_{\pa_t}H|_{t=0}&=-JV\\
\pa^2_tA\big(F_t(\Si)\big)|_{t=0}&=\int_\Si\langle JV,V\rangle\,da_\Si.
\end{align*}
Here $\nabla_{\pa_t}$ refers to the pullback connection on the pullback of
the normal bundle and $J=-\Delta -4-|L|^2$ is the 
Jacobi operator, where $\Delta$ denotes the normal bundle Laplacian.     
\begin{proposition}\label{2varE}
If $\Si$ is a minimal immersed hypersurface in $S^5$, then  
$$
\pa^2_t\cEb\big(F_t(\Si)\big)|_{t=0}=\int_\Si\langle \cJ V,V\rangle\,da_\Si
$$ 
where 
\begin{equation}\label{J}
\cJ=2J(J+4)(J+6).
\end{equation}
\end{proposition}
\begin{proof}
Recall that $\cE$ is given by \eqref{Esphere} and $\cEb=128\,\cE$.  
Since $H|_{t=0}=0$, the $|H|^4$ term does not contribute to the second
variation.  Since 
$|\nabla H|^2-|L|^2|H|^2 +6|H|^2$ vanishes to second order at
$t=0$, we have 
$$
\pa^2_t\cEb\big(F_t(\Si)\big)|_{t=0}
=\int_\Si\pa_t^2|_{t=0}\Big(|\nabla H|^2-|L|^2|H|^2 +6|H|^2 \Big)\,da_\Si
+48\int_\Si\langle J V,V\rangle\,da_\Si.
$$
Again using that $H|_{t=0}=0$, we have 
\[
\begin{split}
\pa_t^2|_{t=0}\Big(|\nabla H|^2-|L|^2|H|^2 +6|H|^2 \Big)
&=2|\nabla \nabla_{\pa_t}H|^2 -2|L|^2|\nabla_{\pa_t}H|^2
+12|\nabla_{\pa_t}H|^2\\
&=2\Big(|\nabla JV|^2-|L|^2|JV|^2+6|JV|^2\Big).
\end{split}
\]
So
\[
\begin{split}
\pa^2_t\cEb\big(F_t(\Si)\big)|_{t=0}&=
2\int_\Si \langle -\Delta JV -|L|^2JV +6JV +24V,JV\rangle\,da_\Si\\
&=2\int_\Si \langle (J+4) JV +6JV +24V,JV\rangle\,da_\Si\\
&=2\int_\Si \langle J(J+4)(J+6)V,V\rangle\,da_\Si.
\end{split}
\]
\end{proof}

Proposition~\ref{2varE} shows that the second variation of $\cEb$ is
determined by the spectral decomposition of the self-adjoint Jacobi
operator $J$.  We 
identify this for $\Si=S^4$ and $\Si=S^2(1/\sqrt{2})\times
S^2(1/\sqrt{2})$. 
The subsequent argument follows closely that of \cite{W}, where the second
variation of the classical Willmore energy at a minimal surface is 
identified.  We refer to \cite{W} for elaboration and proofs of some of the
statements which follow.    

Let $\cK$ denote the 15-dimensional space of Killing fields of $S^5$ and
let $\cK_\Si\subset \Gamma(N\Si)$ denote the space of normal projections 
of restrictions to $\Si$ of elements of $\cK$.  The kernel of 
the restriction-projection map $\cK\rightarrow \cK_\Si$ is the
space of Killing fields whose restriction to $\Si$ is everywhere tangent to 
$\Si$.  Its dimension equals the dimension of the space of isometries of  
$S^5$ which map $\Si$ to itself.  For $\Si=S^4$ this dimension is $10$,
while for $\Si=S^2(1/\sqrt{2})\times S^2(1/\sqrt{2})$ it is $6$.  So 
$\dim \cK_{S^4}=5$ and $\dim \cK_{S^2(1/\sqrt{2})\times 
  S^2(1/\sqrt{2})}=9$.  For any $\Si$, $\cK_{\Si}\subset \ker J$.  

Let $\cC$ denote the 6-dimensional space of tangential projections of
restrictions to $S^5$ of constant vector fields on $\R^6$, and let
$\cC_\Si\subset \Gamma(N\Si)$ 
denote the space of normal projections of restrictions to $\Si$ of elements 
of $\cC$.  Every element of $\cC$ is a conformal Killing field of $S^5$ and
the space of conformal Killing fields of $S^5$ equals $\cK\oplus \cC$.  
The dimension of $\cC_\Si$ is $1$ if $\Si$ is a totally geodesic
$S^4\subset S^5$, and is $6$ otherwise.  For any $\Si$, $\cC_\Si\subset
\ker(J+4)$.  

The space of conformal directions to $\Si$ is $\cK_\Si\oplus \cC_\Si\subset
\Gamma(N\Si)$.  These are in the kernel of $\cJ$ by conformal invariance;
this is consistent with \eqref{J} and the facts that $\cK_{\Si}\subset \ker
J$, $\cC_\Si\subset \ker(J+4)$.  

\bigskip
\noindent
{\it Proof of Proposition~\ref{2varS4}.}  
The normal bundle to $S^4\subset S^5$ has a parallel nonvanishing section.
So its space of sections can be identified with $C^\infty(S^4)$ and
the normal bundle Laplacian with the scalar Laplacian.  The eigenvalues of
$-\Delta$ are $j(j+3)= 0,4,10, \cdots$ with 
multiplicities $\binom{j+4}{4}-\binom{j+2}{4}=1,5,14,\cdots$.  So the 
eigenvalues of $J=-\Delta -4$ are $-4,0,6, \cdots$ with the same
multiplicities.   
By comparing dimensions we see that $\ker J=\cK_{S^4}$ and $\ker
(J+4)=\cC_{S^4}$.  So the kernel of $\cJ$ is exactly the conformal
directions.  Since all other eigenvalues of $J$ are positive, we 
conclude that $\cJ$ is positive transverse to the conformal directions.  
\stopthm

\begin{proposition}\label{2varS2}
The second variation of $\cE$ at $S^2(1/\sqrt{2})\times
S^2(1/\sqrt{2})\subset S^5$ has one  
negative eigendirection, the direction of the family $S^2(r_1)\times
S^2(r_2)$ considered in \S\ref{prod}.  It is positive in all  
eigendirections transverse to this direction and to the 
tangent space to the orbit of the conformal group. 
\end{proposition}
\begin{proof}
The normal bundle to $\Si=S^2(1/\sqrt{2})\times S^2(1/\sqrt{2})\subset S^5$ 
has a parallel nonvanishing section given by \eqref{normal} with
$r_1=r_2$.  So its Laplacian can be identified with two copies of the
scalar Laplacian of $S^2(1/\sqrt{2})$.  The eigenvalues of
$-\Delta_{S^2(1/\sqrt{2})}$ are $2j(j+1)=0,4,12, \cdots$ with  
multiplicities $\binom{j+2}{2}-\binom{j}{2}=2j+1=1,3,5,\cdots$.  So the
eigenvalues of $-\Delta_\Si$ are
$0,4,8,12,\cdots$ with multiplicities $1,6,9,10,\cdots$.  We have $|L|^2=4$
from \eqref{LS2}, so $J=-\Delta-8$.  Hence the eigenvalues of $J$ are
$-8,-4,0,4,\cdots$ with multiplicities $1,6,9,10,\cdots$.  Comparing
dimensions shows that $\ker J =\cK_\Si$ and $\ker(J+4)=\cC_\Si$.  Thus
$\cJ$ has exactly one negative direction, its kernel consists precisely of
the conformal directions, and there is a complementary space to these on
which $\cJ$ is positive.  The $-8$ eigenspace of $J$ is spanned
by constant multiples of $\nu$ from \eqref{normal}.  This is the
variation field of the family $S^2(r_1)\times S^2(r_2)$ analyzed in
\S\ref{prod}.  We noted there that the family had a local maximum at
$r_1=r_2=1/\sqrt{2}$ and it is easily seen from \eqref{tform} that the
second derivative at this maximum is negative.  So for 
$S^2(1/\sqrt{2})\times S^2(1/\sqrt{2})$, this is the only
eigendirection in which $\cE$ decreases.   
\end{proof}

\subsection{Other Energies}\label{otherenergies}
Other conformally invariant energies of $\Si^4\subset
M^n$ can be constructed by adding to $\cE$ conformally invariant
expressions.  The trace-free 
part $\mathring{L}$ scales upon conformal transformation of the metric, so 
$\int_\Si p(\Lo)\,da_\Si$ is conformally invariant for any quartic scalar 
contraction $p(\Lo)$ of the trace-free second fundamental form.  The
following proposition shows that upon 
adding appropriate multiples of $|\Lo|^4$, one obtains non-negative  
energies.  Recall that $\cEb=128\cE$ with $\cE$ given by 
\eqref{Eeuclidean} when $M$ is a Euclidean space.
\begin{proposition}
Suppose $(M,g)$ is $\R^n$ with the Euclidean metric.  If $\beta\geq
\frac43$, then the energy 
\[
\cEb+\beta \int_\Si |\Lo|^4da_\Si
\]
is non-negative.  
\end{proposition}
\begin{proof}
Decomposing 
$L_{\al\be}^{\al'}=\Lo_{\al\be}^{\al'}+\frac14 H^{\al'}g_{\al\be}$ 
gives  
\begin{equation}\label{ELo}
\cEb=\int_\Si \Big( |\nabla H|^2 
-|\Lo^t H|^2 +\frac{3}{16}|H|^4 \Big) 
\,da_\Si.
\end{equation}
Now $\beta |\Lo|^4 -|\Lo^t H|^2 +\frac{3}{16}|H|^4 \geq 0$
since the quadratic form  $\beta x^2 -xy +\frac{3}{16}y^2$ is nonnegative 
for $\beta \geq \frac43$.
\end{proof}

For the subsequent discussion, we denote by $\Wb$ and $\Pb$ the Weyl and 
Schouten tensors of the induced metric on $\Si$.  If $\dim\Si =4$, then
$\int_\Si |\Wb|^2da_\Si$ is conformally invariant.  The Chern-Gauss-Bonnet 
formula in dimension $4$ states
\begin{equation}\label{CGB}
32\pi^2\chi(\Si)=
\int_\Si \big[|\Wb|^2+16\sigma_2(\Pb)\big]\,da_\Si,
\end{equation}
where 
$\sigma_2(\Pb) = \frac12\big((\tr \Pb)^2-|\Pb|^2\big)$ is the second 
elementary symmetric function of the eigenvalues of the Schouten tensor.  
In particular, $\int_\Si \sigma_2(\Pb)\,da_\Si$ is also conformally
invariant.  So further  
conformally invariant energies can be obtained by adding a linear
combination of the integrals of $|\Wb|^2$ and $\sigma_2(\Pb)$.  
The Gauss equation can be used to express $\Wb$ and $\Pb$ in 
terms of $L$ and curvature of $g$.  

In \cite{V}, Vyatkin has used tractor calculus to derive a conformally
invariant energy for $4$-dimensional hypersurfaces in conformally
flat 5-manifolds.  By Theorem 5.2.4 and Lemma 5.2.7 of \cite{V}, his energy
takes the form $\cV=\int_\Si V\,da_\Si$, where    
\[
V = -\tfrac23 \Lo^{\al\be}\na_\al\na^\ga \Lo_{\ga\be}  
-4\Pb_\al{}^\ga \Lo^{\al\be}\Lo_{\ga\be}
+2\Pb_\al{}^\al |\Lo|^2
-\tfrac29 \na^\be\Lo_\be{}^\al \na^\ga \Lo_{\al\ga}.   
\]
Vyatkin's energy can be related to $\cE$ in the case that $M$ is
$5$-dimensional Euclidean space as follows.  The contracted
Codazzi-Mainardi equation gives $\na^\al \Lo_{\al\be}=\frac34 \na_\be H$.   
Integrating by parts one of the derivatives in the first term and then 
combining the first and last terms shows that  
\begin{equation}\label{Vform}
\cV = \int_\Si \Big( \tfrac14 |\nabla H|^2 -4\Pb_\al{}^\ga
\Lo^{\al\be}\Lo_{\ga\be} +2\Pb_\al{}^\al |\Lo|^2\Big)\,da_\Si.
\end{equation}
The intrinsic Schouten tensor $\Pb$ can be expressed as a linear
combination of quadratic terms in $L$ via the Gauss equation.  Doing so,
substituting and simplifying gives
\begin{equation}\label{Vterms}
-4\Pb_\al{}^\ga \Lo^{\al\be}\Lo_{\ga\be} +2\Pb_\al{}^\al |\Lo|^2
=2\tr \Lo^4 -\tfrac23 |\Lo|^4-H\tr\Lo^3 +\tfrac18 H^2|\Lo|^2.
\end{equation}
Similar calculations express $|\Wb|^2$ and $\sigma_2(\Pb)$ in terms of $L$:  
\begin{gather}\label{sigma2W}
\begin{gathered}
|\Wb|^2=\tfrac73 |\Lo|^4-4\tr \Lo^4\\
4\sigma_2(\Pb)=\tr\Lo^4-\tfrac13 |\Lo|^4 -H\tr\Lo^3
+\tfrac38 H^2 |\Lo|^2-\tfrac{3}{64} H^4.
\end{gathered}
\end{gather}
Combining \eqref{ELo}, \eqref{CGB}, \eqref{Vform}, \eqref{Vterms}, and
\eqref{sigma2W} gives 
\[
\cV= \tfrac14 \cEb +8\pi^2\chi(\Si)+ \int_\Si \Big(2\tr\Lo^4-\tfrac{11}{12} 
|\Lo|^4\Big) \,da_\Si.
\]

\end{document}